\documentclass[a4paper, 11pt]{article}

\usepackage{amsfonts,amsbsy,amsmath}
\usepackage{enumerate,graphics,psfrag,epsfig}

\newtheorem{proposition}{Proposition}[section]
\newtheorem{theorem}[proposition]{Theorem}
\newtheorem{lemma}[proposition]{Lemma}

\newtheorem{definition}[proposition]{Definition}

\newtheorem{remark}[proposition]{Remark}

\newenvironment{proof}{\smallskip\noindent\emph{\textbf{Proof.}}\hspace{1pt}}%
{\hspace{-5pt}{\nobreak\quad\nobreak\hfill\nobreak$\square$\vspace{8pt}%
\par}\smallskip\goodbreak}

\newcommand{\Section}[1]{\section{#1}\setcounter{equation}{0}}

\newcommand{\modulo}[1]{{\left|#1\right|}}
\newcommand{\norma}[1]{{\left\|#1\right\|}}

\newcommand{\reali}{{\mathbb{R}}}

\newcommand{\BV}{\mathbf{BV}}

\renewcommand{\epsilon}{\varepsilon}
\renewcommand{\phi}{\varphi}
\renewcommand{\L}[1]{\mathbf{L^#1}}

\newcommand{\sign}{\mathop{\mathrm{sign}}}

\title{ Modeling and analysis of pooled stepped chutes }

\author{G.~Guerra\footnote{Universit\`a degli Studi di Milano--Bicocca,
    I-20125 Milano, Italy} \and 
  M.~Herty\footnote{RWTH Aachen University, D-52074 Aachen,
    Germany.} \and F.~Marcellini$^*$}

\begin{document}

\maketitle

\begin{abstract}
  We consider an application of pooled stepped chutes where the
  transport in each pooled step is described by the shallow--water
  equations. Such systems can be found for example at large dams in
  order to release overflowing water. We analyze the mathematical
  conditions coupling the flows between different chutes taken from the
  engineering literature.  We present the solution to a Riemann
  problem in the large and also a well--posedness result for the
  coupled problem. We finally report on some numerical experiments.

  \noindent\textit{2000~Mathematics Subject Classification:} 35L65.

  \medskip

\end{abstract}

\Section{Introduction}
\label{sec:Intro}
This work deals with water behavior of so--called pooled stepped
chutes. This geometry frequently appears in the real water dams and it
also appears in mountain rivers to control the bed load transport. In
both cases the main concern is to spill excessive floodwater in
additional channels next to the dam structure. These are called pooled
stepped chutes or pooled steps.  Within the pooled steps additional
weirs perpendicular to the flow direction are introduced to increase
energy dissipation. This problem has been gained some attention in
recent years in the engineering community, see e.g.
\cite{bradley1957,chamani1994,chanson1994,chanson1994b,chanson2002,ohtsu2004,
sorenson1985,IWW1}.  However, only a few mathematical discussions are
currently available \cite{dressler1949,rouse1936}.  In particular, the
modeling and design of the spillways and stepped channels have so far
been addressed using experiments and data fitting techniques, see
e.g. \cite{IWW1}.  From the measurements empirical formulas have been
derived and used in sophisticated simulations. The measurements taking
into account complex geometries as well as material properties and the
air--water mixture leading to a variety of different empirical
formulas and tables \cite{elkamash2005}.
 \par So far, the mathematical discussion has been limited to a
consideration of the effect of the weir at the end of a spillway {\em
neglecting } the dynamics of the water inside the pooled
channels. Here, we discuss the mathematical implications of
considering the coupled problem, i.e., the dynamics inside the pooled
steps and the (empirical or theoretical) conditions imposed closed to
the weir.  Typically, the water flow in the channels is described by
the shallow--water equations whereas the effect of the weir is given
by some algebraic condition. We treat this problem using a network
approach with the conditions at the weir as coupling conditions. We
present a well--posedness result for a simple condition based on
energy dissipation.  The recent literature offers several results on
the modeling of systems governed by conservation laws on networks. For
instance, in~\cite{BandaHertyKlar2, BandaHertyKlar1,
ColomboGaravello3,MR2568708} the modeling of a network of gas
pipelines is considered. The basic model is the $p$-system or,
in~\cite{ColomboMauri,MR2529960}, the full set of Euler equations. The
key problem in these papers is the description of the evolution of
fluid at a junction between two or more pipes.  A different physical
problem, leading to a similar analytical framework, is that of the
flow of water in open channels, considered for example
in~\cite{LeugeringSchmidt}.

Consider a water flow in an open canal affected by a weir or small dam
at the point $x=0$. If the water level becomes greater than the height
of the weir, some water passes over the weir. Similarly to the models
in~\cite{CoronBastinEtAl, Gugat, LeugeringSchmidt} we describe the
dynamics of the water by the shallow water equations, while the
interaction with the weir is described by coupling conditions.

Let $(h,v)(t,x)$ be respectively
the water level (with respect to the flat bottom) and its velocity
for $x\not=0$. The shallow--water equations in each canal 
are given by 
\begin{equation}
  \label{eq:SH}
  \begin{cases}
    \partial_t h + \partial_x (hv) = 0\\
    \partial_t (hv) + \partial_x\left(h{v}^2 +
      \frac 12 g{h}^2\right) = 0
  \end{cases}
  \quad x \not= 0,
\end{equation}
where $g$ is the gravity constant. Two canals are coupled by so--called
pooled steps \cite{IWW1}. A sketch of this situation is given in Figure \ref{fig:diga}. 
\begin{figure}[h]
  \centering
  \includegraphics[width=8cm]{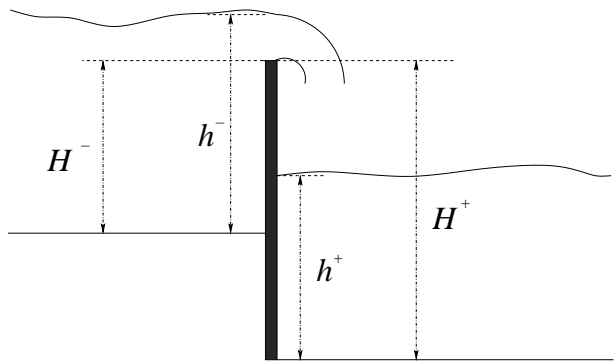}
  \caption{This figure represents two connected pooled steps with a weir in between and
    the indication of the various heights}
  \label{fig:diga}
\end{figure}
In the engineering literature \cite{chanson2002,OpenChHy,IWW1} (see also Remark
\ref{remark1})
 water of height $h^{-}$  flowing over a weir of height $H^{-}$ generates a flow $Q$
  \begin{equation}\label{32law}
 Q = \tilde C \; \sqrt{g} \left( h^{-} - H^{-} \right)^{3/2},
 \end{equation}
where $\tilde C$ is a constant depending on the air--water--ratio and
the detailed geometry. In many situations $\tilde C=0.6$ is used. The
equation (\ref{32law}) is called $3/2-$law. Formally, it can be
derived from the following idea: the potential energy of the water
flowing over the weir is transformed to kinetic energy. Hence, we have
the balance
$$ (\rho h)  \frac{1}2 g h = \frac{1}2 (\rho h) v^{2}.$$
Solving for $v$ in terms of $h$ we obtain the previous formula and
$1-\tilde C$ is the percentage of energy loss during the change of
potential to kinetic energy. We use this idea to deduce the coupling
condition for (\ref{eq:SH}) at $x=0.$ We conserve the total water over
the weir and hence $h^{+}v^{+}=h^{-}v^{-}.$ The difference in the
amount of water overflowing the weir is $\left[h^--H^-\right]_+ -
\left[h^+-H^+\right]_+$. This defines the velocity (and its sign) at
the weir according to the balance of potential and kinetic
energy. Hence, the coupling conditions are
\begin{equation}
  \label{eq:coupl}
  \begin{cases}
      h^-v^- = C\left(\left[h^--H^-\right]_+ -
        \left[h^+-H^+\right]_+\right)
      \cdot
      \sqrt{\left[h^--H^-\right]_+ -
        \left[h^+-H^+\right]_+}
    \\[10pt]
    h^+v^+ = h^-v^-,
  \end{cases}
\end{equation}
where $C=0.6\sqrt{g}$, $(h^-,v^-)=(h,v)(t,0-)$, 
$(h^+,v^+)=(h,v)(t,0+)$ while $H^\pm$ are the heights of the weir
to the left and the right (see Figure \ref{fig:diga}). 

\begin{remark}\label{remark1}
Obviously, (\ref{32law}) is only a first approximation on the complex
dynamics at the weir, see
\cite{chanson1994,rouse1936,IWW1}.
As outlined in the introduction there exists a variety of empirical
formulas in the engineering community. Many of them include further
effects as for example the water--air ratio of the overspill or the
roughness of the channel bottom. For example in \cite[Equation
7.7]{IWW1} the following relation has been determined
\begin{eqnarray*}
  Q = \left( h^{- } - H^{-} \right)^{3/2} \left( \frac{2}3 \mu \sqrt{2 g} \right)
    = \\
 C_{1} \sqrt{g}  \left( h^{- } - H^{-} \right)^{3/2} + C_{2}\sqrt{g}
 \left( h^{- } - H^{-} \right)^{5/2}/ H^-.
\end{eqnarray*}
Here, we have $\mu=0.611 + 0.08 (h^{-}- H^{-}) / H^{-}$ and the constants  are
$C_{1}=\frac{2}{3} \; \sqrt{2}\; 0.611$ and $C_{2}= \frac{2}{3}\;\sqrt{2} \;
0.08$.
Since the coefficient $C_{1}$ is
roughly eight times larger than $C_{2},$ this equation is very similar
to (\ref{32law}).
Another example is given in
\cite{Blaisdell1954},
\cite[Equation 2.50--2.51]{IWW1}, where  the following formula has
been proposed for the flow with $H=h^{-}-H^{-}$
\begin{eqnarray*}
Q = v^{-} H (k_{c}+k_{d})  =  0.15 - 0.45 \frac{Ê(v^{-})^{2}}{Ê2 g H } + \\
\left( 0.57 - 2 \left( \frac{ (v^{-})^{2}}{ 2 g H} - 0.21 \right)^{2}
  \exp\left( 10 \left( \frac{ (v^{-})^{2}}{ 2 g H} - 0.21 \right) \right)\right).
\end{eqnarray*}
\par
The case of no weir, i.e.,  $H^{-} \equiv 0$ and $H^{+} > 0$,  lead to the following
studied formulas,
 e.g. \cite{rouse1936}, \cite[Equation 2.38]{IWW1}
 $$ Q = v^{-} 0.715 h^{-}$$
 or \cite{rand1955}, \cite[Equation 2.39--2.42]{IWW1}
 $$ Q = v^{-}  H^{+}   \left(  h^{-} / H^{+} \right)^{1.275}.$$
 We restrict our discussion to the still commonly used (\ref{32law}) to outline the
 ideas. Further note that 
 additional empirical formulas for the arising wave in the outgoing pooled step are
 not needed in our approach, since these dynamics are fully covered by
 the shallow--water
 equation. 
\end{remark}


\section{The Riemann problem for a single weir}
\label{sec:RP}
By Riemann Problem at the weir we define the problem \eqref{eq:SH}, \eqref{eq:coupl}
with initial data
\begin{equation}
  \label{eq:riemdata}
  (h,v)=
  \begin{cases}
    (h_l,v_l) & \text{ for }x<0\\
    (h_r,v_r) & \text{ for }x>0.\\    
  \end{cases}
\end{equation}
\begin{definition}
  \label{def:rp}
A solution to the Riemann
Problem \eqref{eq:SH}, \eqref{eq:coupl}, \eqref{eq:riemdata}
is a function $(h,v): \reali^+ \times \reali\to \reali^2$ such that
$(t, x) \to (h,v)(t, x)$ is self-similar and coincides in $x>0$
with the restriction of the Lax solution to the standard Riemann Problem
for \eqref{eq:SH} with initial data
\begin{equation}
  \label{eq:riemdata1}
  (h,v)=
  \begin{cases}
    (h,v)(t,0+) & \text{ for }x<0\\
    (h_r,v_r) & \text{ for }x>0.\\    
  \end{cases}
\end{equation}
while coincides in $x<0$
with the restriction to of the Lax solution to the standard Riemann Problem
for \eqref{eq:SH} with initial data
\begin{equation}
  \label{eq:riemdata2}
  (h,v)=
  \begin{cases}
    (h_l,v_l) & \text{ for }x<0\\
    (h,v)(t,0-) & \text{ for }x>0.\\    
  \end{cases}
\end{equation}
moreover $(h,v)(t,0\pm)=(h^\pm,v^\pm)$ satisfy \eqref{eq:coupl}.
\end{definition}

For studying the Riemann problem we first collect the standard
expressions for the eigenvalues,
eigenvectors and Lax curves for the shallow water equations \eqref{eq:SH}.

The $2\times 2$ system of conservation laws in~\eqref{eq:SH} has
  the eigenvalues $\lambda_1,\,\lambda_2$ and the eigenvectors
  $r_1,\,r_2$, where
  \begin{equation}
    \label{eq:eigenvalues2}
    \begin{array}{r@{\;}c@{\;}l@{\qquad}r@{\;}c@{\;}l}
      \lambda_1 (h, v) & = & v-\sqrt{gh}
      &
      \lambda_2 (h, v) & = & v+\sqrt{gh}
      \\[6pt]
      r_1(h, v) & = &
      \left[
        \begin{array}{c}
          -1 \\
          -v + \sqrt{gh}
        \end{array}
      \right]
      &
      r_2(h, v)
      & = &
      \left[
        \begin{array}{c}
          1 \\
          v + \sqrt{gh}
        \end{array}\right]
    \end{array}
  \end{equation}
  \begin{displaymath}
      \nabla_{(h,hv)}\lambda_{1}(h, v) \cdot r_{1}(h, v)
      \frac{3}{2}\sqrt{\frac gh}>0\quad
      \nabla_{(h,hv)}\lambda_{2}(h, v) \cdot r_{2}(h, v)
      \frac{3}{2}\sqrt{\frac gh}>0.
  \end{displaymath}
  The Lax curves of the first and second family, described in
  Figure~\ref{fig:CurveLax}, are:
  \begin{equation}
    \label{Eq:Lax12x2}
    v = \mathcal{L}_1^+(h;h_0, v_0)=
    \begin{cases}
      v_0 - 2\left(\sqrt{gh}-\sqrt{gh_0}\right) & h\leq h_0
      \\[6pt]
      v_0 - (h-h_0) \sqrt{\frac 12 g \frac{h + h_0}{hh_0}} & h>h_0,
    \end{cases}
  \end{equation}
  \begin{equation}
    \label{Eq:Lax32x2}
    v = \mathcal{L}_2^+(h;h_0, v_0)
    =
    \begin{cases}
      v_0 + 2\left(\sqrt{gh}-\sqrt{gh_0}\right) & h\geq
      h_0\\[6pt]
      v_0 + (h-h_0) \sqrt{\frac 12 g \frac{h + h_0}{hh_0}} & h<h_0.
    \end{cases}
  \end{equation}
  The reversed Lax curves of the first and second family
  are given by:
  \begin{equation}
    \label{Eq:Lax12x2r}
    v = \mathcal{L}_1^-(h;h_0, v_0)=
    \begin{cases}
      v_0 - 2\left(\sqrt{gh}-\sqrt{gh_0}\right) & h\geq h_0
      \\[6pt]
      v_0 - (h-h_0) \sqrt{\frac 12 g \frac{h + h_0}{hh_0}} & h<h_0,
    \end{cases}
  \end{equation}
  \begin{equation}
    \label{Eq:Lax32x2r}
    v = \mathcal{L}_2^-(h;h_0, v_0)
    =
    \begin{cases}
      v_0 + 2\left(\sqrt{gh}-\sqrt{gh_0}\right) & h\leq
      h_0\\[6pt]
      v_0 + (h-h_0) \sqrt{\frac 12 g \frac{h + h_0}{hh_0}} & h>h_0.
    \end{cases}
  \end{equation}
  \begin{figure}[h]
    \centering
    \includegraphics[width=8cm]{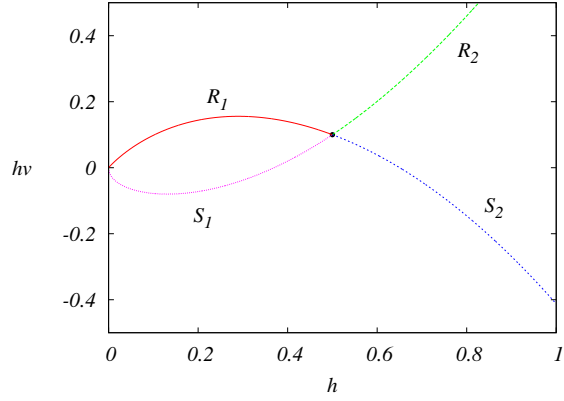}
    \caption{Lax curves for the shallow water system in the $(h,hv)$ plane}
  \label{fig:CurveLax}
\end{figure}
Concerning the coupling condition, we study the states which satisfy it at the
point $x=0$. We restrict ourselves to the subsonic region $|v| < \sqrt{gh}$.
From \eqref{eq:coupl}, since the heights are
always positive, it follows
\begin{displaymath}
  \sign v^- = \sign \left(\left[h^--H^-\right]_+ -
        \left[h^+-H^+\right]_+\right).
\end{displaymath}
Moreover
\begin{displaymath}
   h^-\left|v^-\right| = C\left|\left[h^--H^-\right]_+ -
        \left[h^+-H^+\right]_+\right|^{\frac 32}
\end{displaymath}
which becomes
\begin{displaymath}
  \left(\frac{h^-\left|v^-\right|}{C}\right)^{\frac 23}
  = \left(\left[h^--H^-\right]_+ -
        \left[h^+-H^+\right]_+\right)\cdot\sign v^-,
\end{displaymath}
\begin{displaymath}
  \left[h^+-H^+\right]_+ = \left[h^--H^-\right]_+ -
  \left(\frac{h^-\left|v^-\right|}{C}\right)^{\frac 23}
  \cdot\sign v^-.
\end{displaymath}
We have also to add the equality of the fluxes, hence we obtain
\begin{equation}
  \label{eq:coup2}
  \begin{cases}
  \left[h^+-H^+\right]_+ = \left[h^--H^-\right]_+ -
  \left(\frac{h^-\left|v^-\right|}{C}\right)^{\frac 23}
  \cdot\sign v^-\\
  h^+v^+ = h^-v^-.  
  \end{cases}
\end{equation}
Consider first the case where the water level to the right is below
the weir, $h^+\le H^+$ so that $ \left[h^+-H^+\right]_+ =0$.
If also $h^-<H^-$ then no water crosses the weir, while if
$h^->H^-$ the water
crosses the weir flowing from left to right. The
left state must satisfy $v^-\ge 0$ and
\begin{displaymath}
  \left[h^--H^-\right]_+ -
  \left(\frac{h^-\left|v^-\right|}{C}\right)^{\frac 23}
  = 0,
\end{displaymath}
which can be rewritten as
\begin{equation}
  \label{eq:uppercurve}
  h^-v^-=C\left[h^--H^-\right]_+^{\frac 32}. 
\end{equation}
The curve $(h^-,v^-)$ where $h^-,\,v^-$ satisfy \eqref{eq:uppercurve}
consists of all the left states which can be connected to a
right state with $h^+\le H^+$ (see Figure \ref{fig:CurvaD}). We call
the support of this curve $\Gamma_u$. The case $h^+\ge H^+$, $h^-<H^-$
is symmetric, hence we call $\Gamma_l$ the support of the curve 
\begin{displaymath}
  \label{eq:lowercurve}
  -h^+v^+=C\left[h^+-H^+\right]_+^{\frac 32}.   
\end{displaymath}
Since $C\left[h^--H^-\right]_+^{\frac 32}=
\tilde C\sqrt{g}\left[h^--H^-\right]_+^{\frac 32}\le \sqrt{g}
\left({h^-}\right)^{\frac 32}$ both $\Gamma_u$ and $\Gamma_l$ lye in the
subsonic region. If both levels are above
the weir, then \eqref{eq:coup2} gives a unique state $(h^+,v^+)$
which connects a
given state $(h^-,v^-)$. We call this function $(h^+,h^+v^+)=\Phi (h^-,h^-v^-)$.
We have the following lemma.

\begin{lemma}
  \label{lemma:cons}
  Consider the $(h,hv)$ plane, and define the following sets depicted in Figure \ref{fig:CurvaD}:
  \begin{displaymath}
    \begin{split}
      \Omega_u&=\left\{(h,hv): hv > C\left[h-H^-\right]_+^{\frac{3}{2}}\right\},\\
      \Sigma_l&=\left\{(h,hv): v < 0,\; 0<h \le H^- \right\},\\
      A^-&=\left\{(h,hv): 0< hv < C\left[h-H^-\right]_+^{\frac{3}{2}}\right\},\\
      B^-&=\left\{(h,hv): v \le 0,\; h > H^- \right\},\\
      \Omega_l&=\left\{(h,hv): hv < -C\left[h-H^+\right]_+^{\frac{3}{2}}\right\},\\
      \Sigma_u&=\left\{(h,hv): v > 0,\; 0<h \le H^+ \right\},\\
      A^+&=\left\{(h,hv): v > 0,\; h > H^+ \right\},\\
      B^+&=\left\{(h,hv): -C\left[h-H^+\right]_+^{\frac{3}{2}} < hv \le 0\right\}.
    \end{split}
  \end{displaymath}
  Then, we have that  the coupling condition induces the following assertions:
  \begin{itemize}
  \item 
    no left state $(h^-,h^-v^-)\in \Omega_u$ can be connected to any
    right state $(h^+,h^+v^+)$ and no right state $(h^+,h^+v^+)\in
    \Omega_l$ can be connected to any left state $(h^-,h^-v^-)$;
  \item
    any left state in $\Gamma_u$ can be connected to any right state in
    $\Sigma_u$ with the same flux: $h^-v^-=h^+v^+$;
  \item
    any right state in $\Gamma_l$ can be connected to any left state in
    $\Sigma_l$ with the same flux: $h^-v^-=h^+v^+$;
  \item
    any left state in $A^-$ can be connected to one and only one right state
    in $A^+$ and any right state in $A^+$ can be connected to one and only one
    left state
    in $A^-$;
  \item
    any left state in $B^-$ can be connected to one and only one right state
    in $B^+$ and any right state in $B^+$ can be connected to one and only one left state
    in $B^-$.     
  \end{itemize}
\end{lemma}
\begin{figure}[h]
  \centering
  \includegraphics[width=6.2cm]{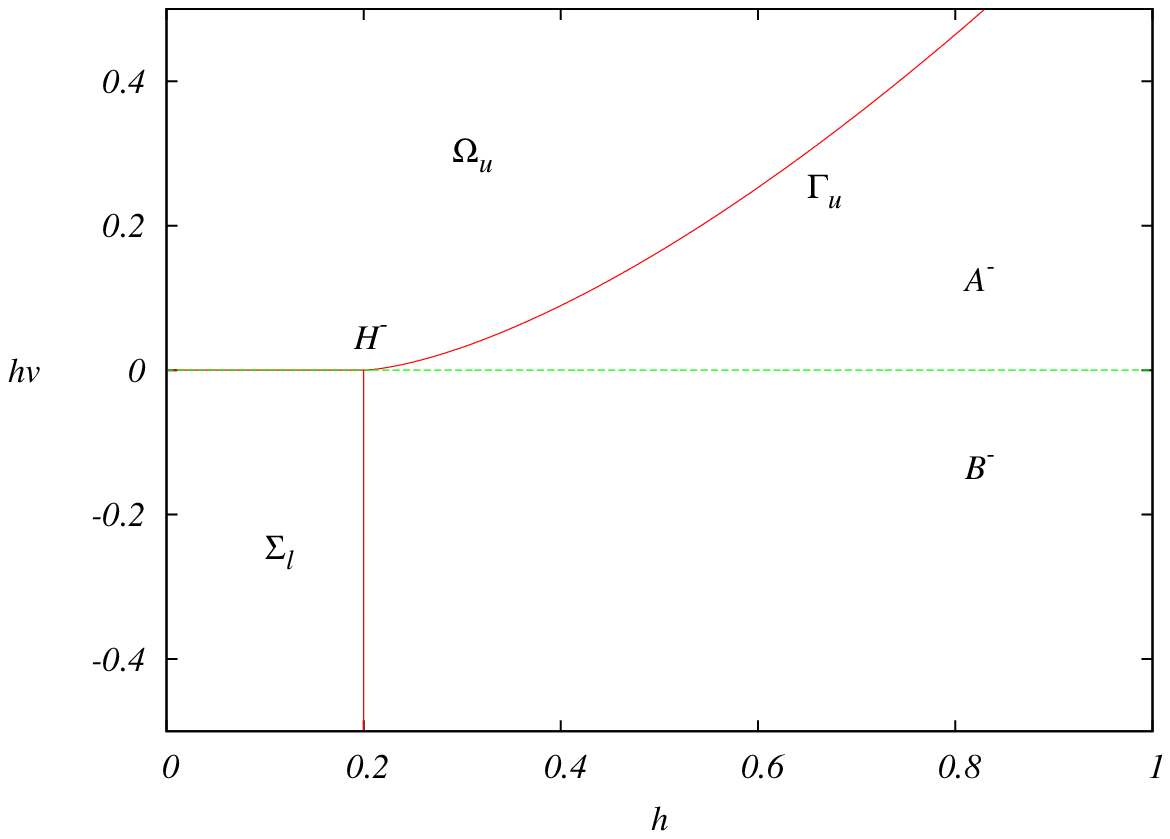}
  \includegraphics[width=6.2cm]{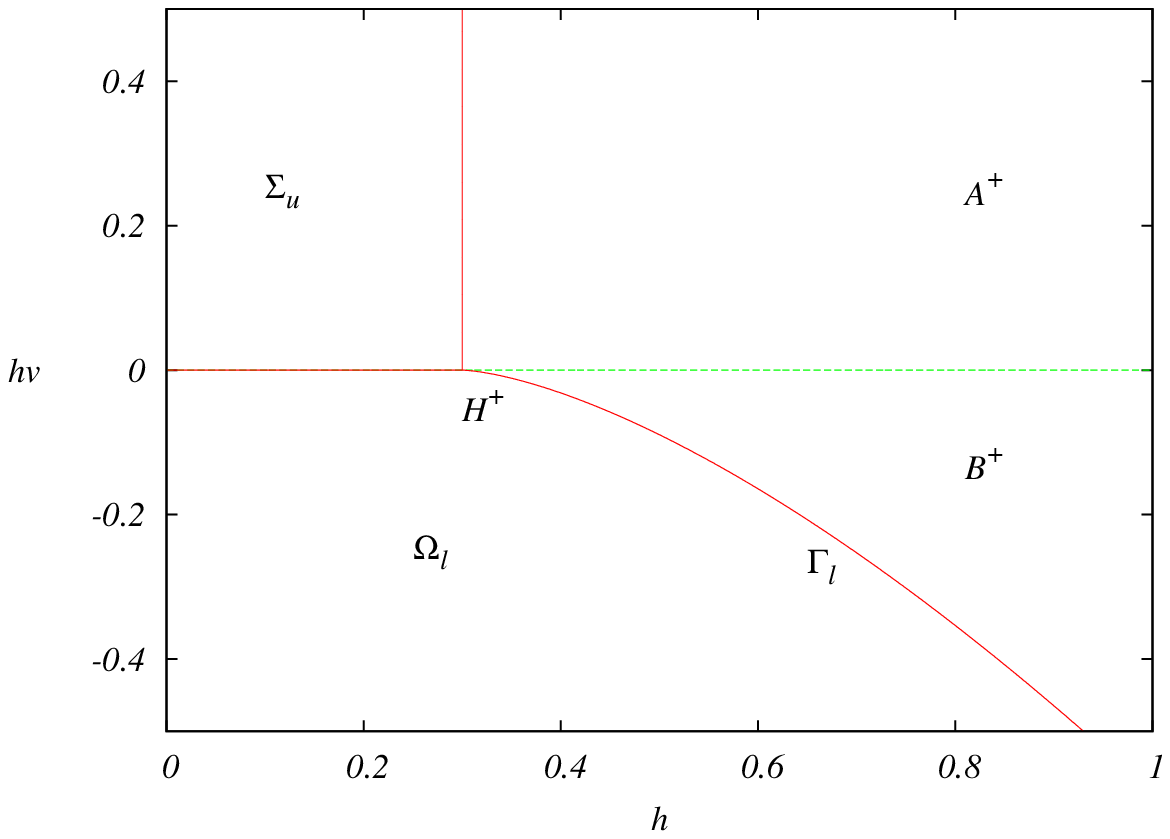}
  \caption{Representation of the regions introduced in Lemma \ref{lemma:cons}}
  \label{fig:CurvaD}
\end{figure}

We omit the proof since it is a straightforward analysis of the coupling conditions
\eqref{eq:coup2}.




We call $\Phi: A^-\cup B^-\to A^+\cup B^+$ the one to one map that associates
to any left state in $A^-\cup B^-$ the corresponding
right state in $A^+\cup B^+$ connected
through the coupling condition:
\begin{displaymath}
  \Phi(h,hv)=\left(H^++h-H^--\left(\frac{h|v|}{C}\right)^{\frac{2}{3}}
    \sign v,hv\right). 
\end{displaymath}

We also define
\begin{displaymath}
  \Phi(h,hv)=\left(\left(-vh/C\right)^{\frac{2}{3}}+H^+,hv\right)\in\Gamma_l,
  \quad \text{for any }(h,hv)\in\Sigma_l,
\end{displaymath}
in such a way that the left state $(h^+,h^+v^+)=\Phi(h^-,h^-v^-)$
is connected through the
coupling condition to the right state $(h^-,h^-v^-)$ for any\\
$(h^-,h^-v^-)\in\Sigma_u\cup A^-\cup B^- $ .

We are now able to show that the Riemann problem can be solved in the large.

\begin{proposition}
  For all states $(h_l,v_l)$ and $(h_r,v_r)$ in the subsonic region,
  the Riemann problem \eqref{eq:SH}, \eqref{eq:coupl}, \eqref{eq:riemdata}
  has a unique
  solution satisfying Definition \ref{def:rp}, whose constant states
  depend continuously on the initial data and are constructed glueing together
  a wave of the first family, the coupling condition and a wave of the second
  family.
\end{proposition}

\begin{proof}
  Given two states $(h_l,v_l)$ and $(h_r,v_r)$ in the subsonic region,
  we proceed in the following way for solving the Riemann
  problem. Draw the Lax curve of the
  first family in the $(h,hv)$ plane through $(h_l,v_l)$. Since, in the subsonic
  region, it is
  strictly decreasing,
  it intersects in one and only one point
  (with $h>0$) the convex and non decreasing curve
  $\Gamma_u$.
  We call the unique point of intersection $(h^*,v^*)$.  
  Then, we 
  consider the non increasing curve
  \begin{equation}
    \label{eq:connect}
    \gamma(h) =
    \begin{cases}
      \left(\frac{H^+}{h^*}h,h^* v^*\right)   & \text{for }h \le h^*\\
      \Phi\left(h,h\mathcal{L}^+_1(h;h_l,v_l)\right) & \text{for }h >
      h^*.
    \end{cases}
  \end{equation}
  Next, we take the
  inverse Lax curve related to the second characteristic field passing
  through the right state $(h_r,v_r)$. This curve is strictly increasing
  in the subsonic region and in the upper supersonic region,
  therefore it has one and only one point of intersection with
  $\gamma$ (for positive $h$), denoted by   $(k,w)$. This point might belong to 
  the upper supersonic region.
  The Riemann problem is finally solved in
  the following way: take the subsonic (or lower supersonic)
  state $(k^*,w^*)$ on the first Lax
  curve such that $kw = k^*w^*$ and connect
  $(h_l,v_l)$ to $(k^*,w^*)$ with a wave of
  the first family. This wave has negative velocity since the curve 
  belongs entirely to
  the subsonic region or to the subsonic region and the lower
  supersonic one. Then, the points $(k^*,w^*)$ and $(k,w)$
  satisfy the coupling conditions \eqref{eq:coupl}. 
   Finally $(k,w)$ and $(h_r,v_r)$ are connected by
  a wave of the second family which travels with positive velocity even if
  $(k,w)$ happens to be in the upper supersonic region, see Figure
  \ref{fig:rimprob}.
\end{proof}
\begin{figure}[h]
  \centering
  \includegraphics[width=6.2cm]{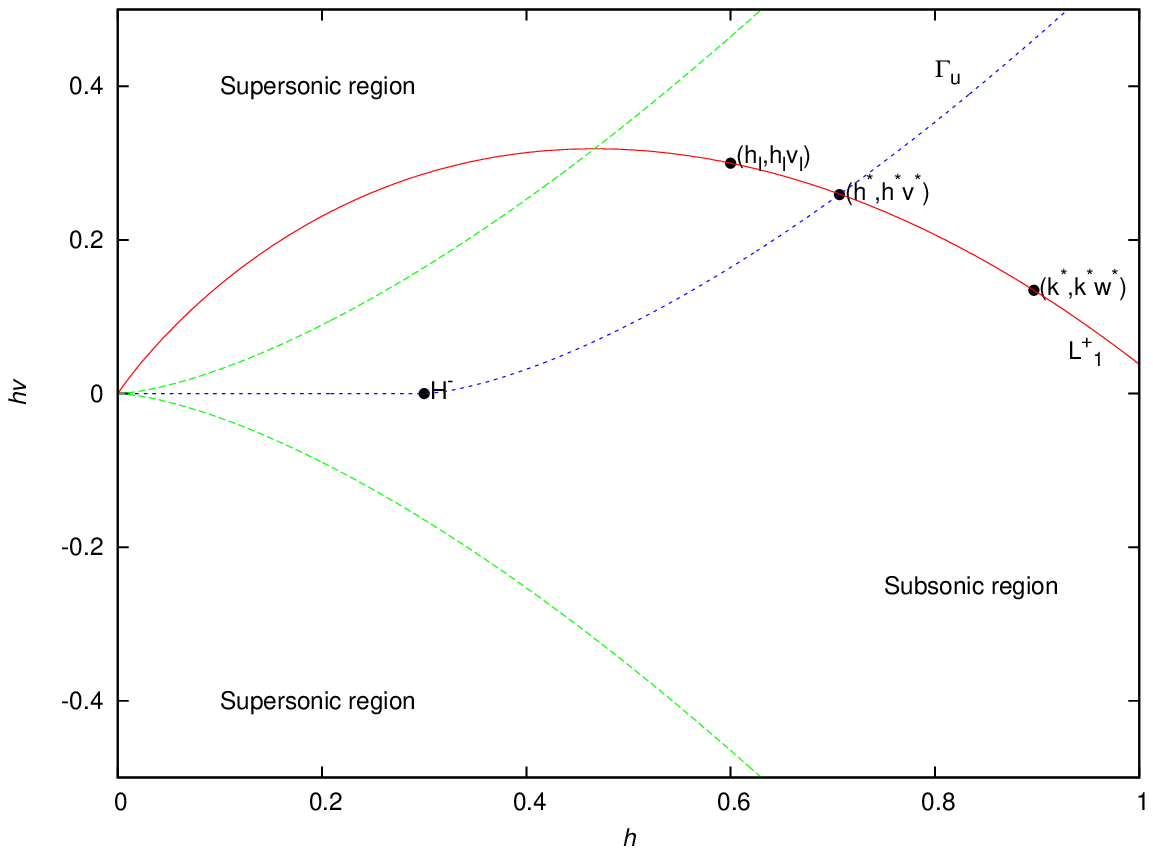}
  \includegraphics[width=6.2cm]{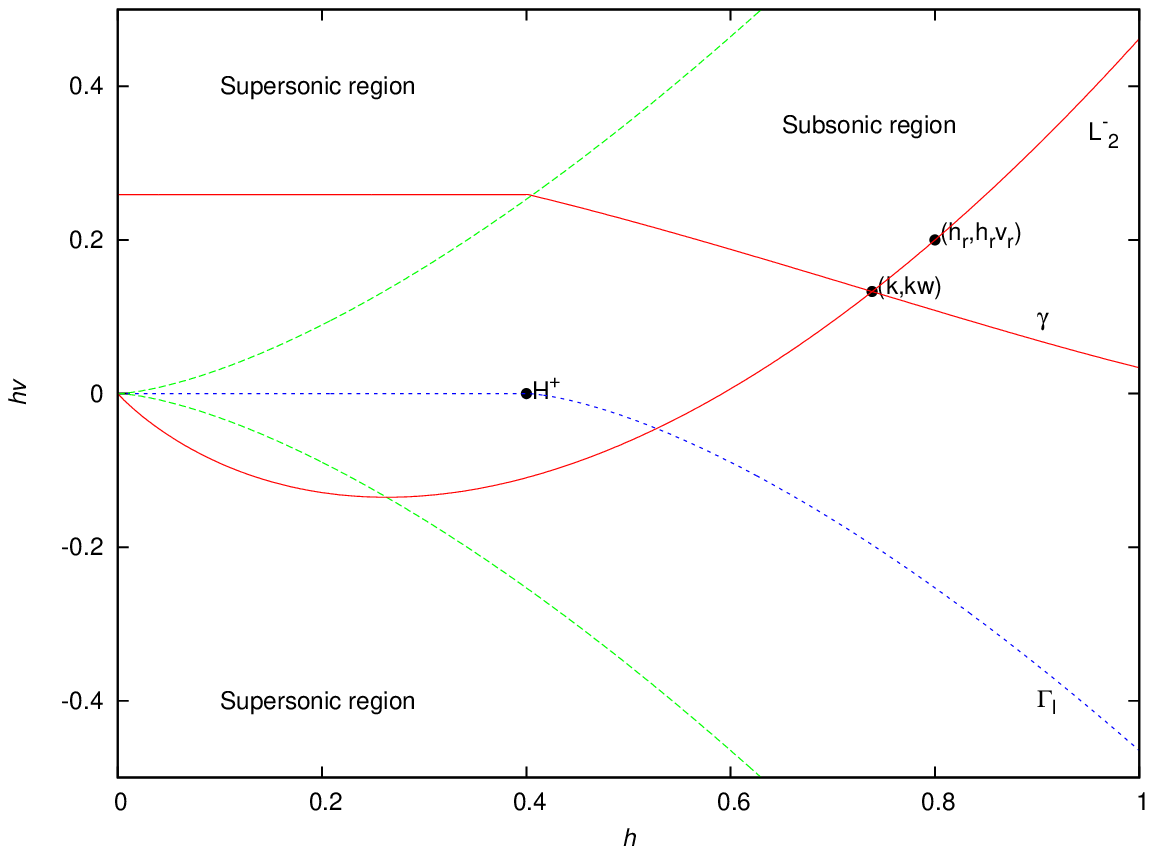}\\
  \includegraphics[width=6.2cm]{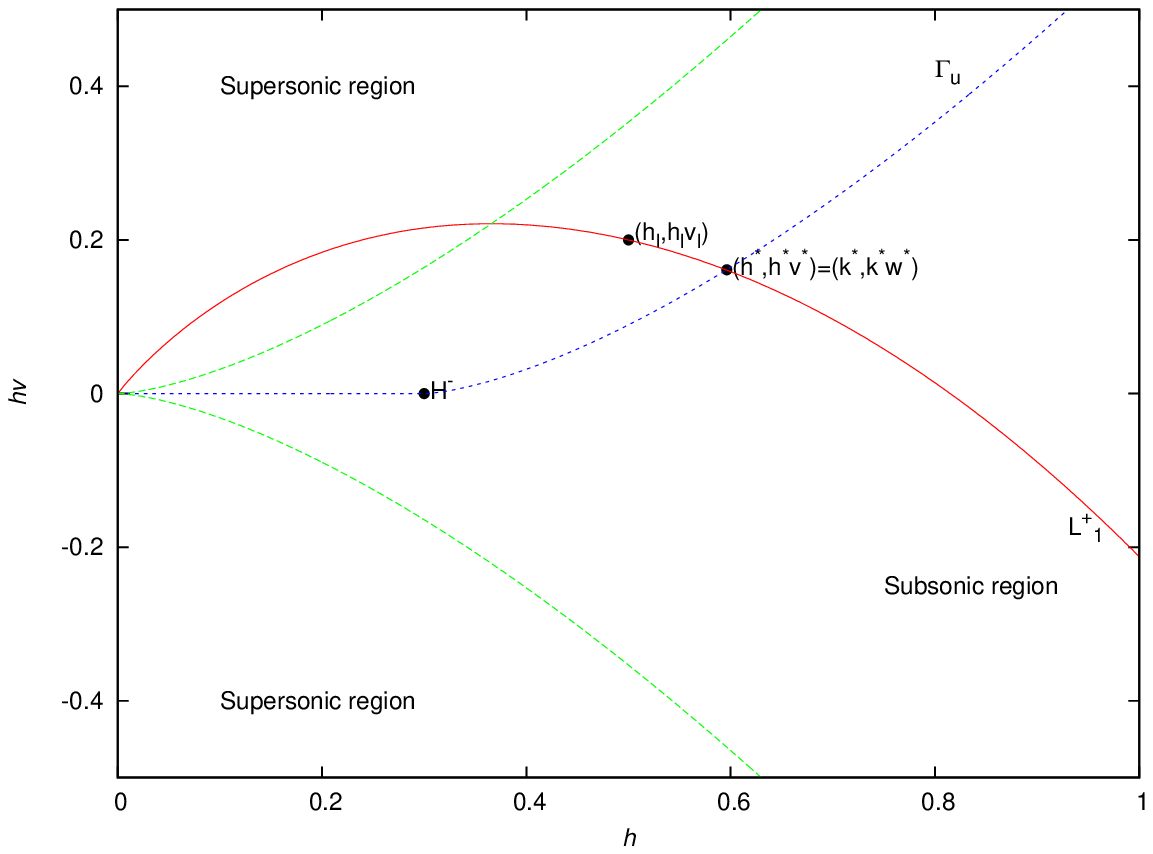}
  \includegraphics[width=6.2cm]{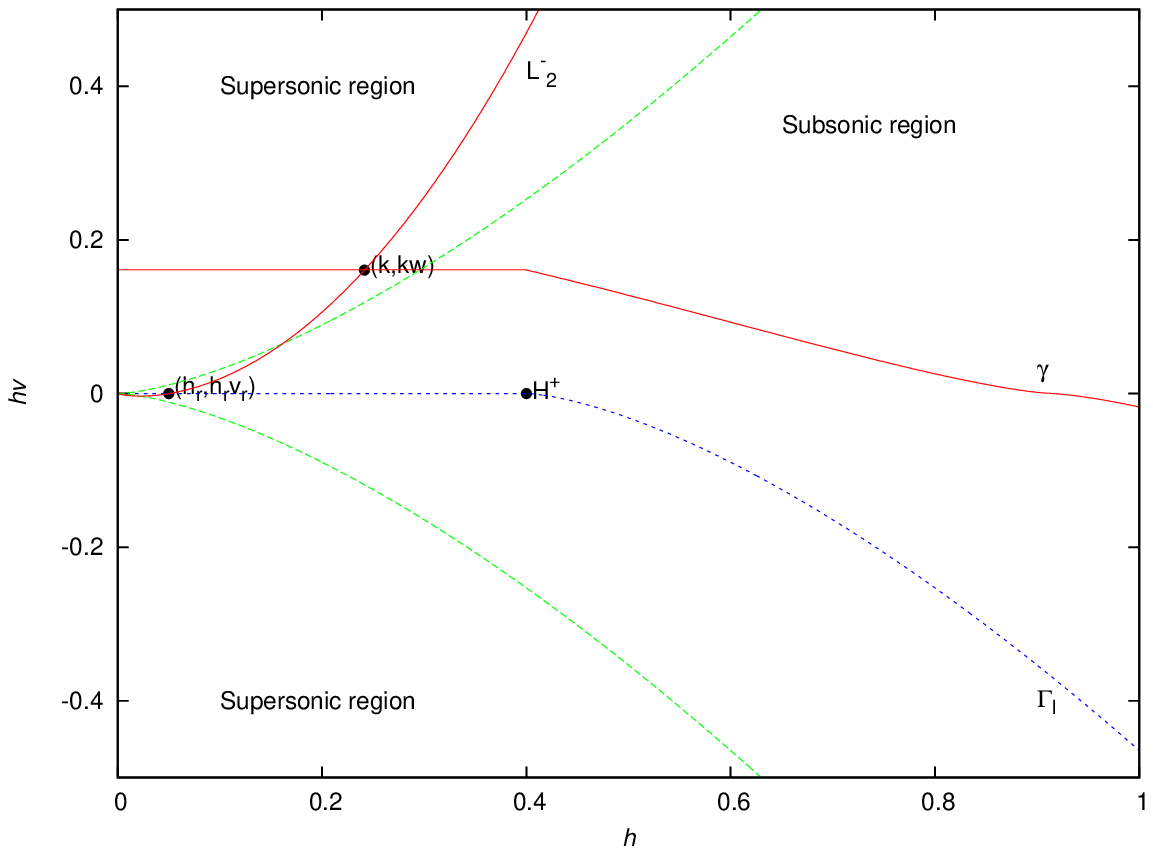}
  \caption{These figures describe how to solve two Riemann Problem (first and second row).
    The states on the left and on the right of the weir are represented respectively
    on the left and right figures. The green lines separate the subsonic and supersonic
    regions}
  \label{fig:rimprob}
\end{figure}

\section{A well posedness result}
This section is devoted to show the well--posedness of the Cauchy problem
for system \eqref{eq:SH} for initial data around a constant subsonic state
satisfying the coupling condition \eqref{eq:coupl}.
For simplicity we suppose that the water level to the right of the weir is below
the weir level, that is $h^+ < H^+$. This implies that the water can only flow
from the left to the right of the weir when its left level overflows
the weir.
\begin{theorem}\label{mainthm}
  Given $H_0^-,\;H^+$ and
  two constant states in the subsonic region $(h_0,v_0)$ and $(h_1,v_1)$
  such that
  \begin{equation}
    \label{eq:theomhy}
    h_0v_0=C\left[h_0 - H^-_0\right]_+^{\frac 32},\quad h_0v_0 = h_1v_1,\quad
    h_0 > H_0^-,\quad h_1 < H^+,
  \end{equation}
  then, there exists a closed domain
  \begin{displaymath}
    \mathcal{D} \subseteq \left\{ (h,v) \in (h_0,v_0)\chi_{(-\infty,0)}
      + (h_1,v_1)\chi_{(0,\infty)} +(\L1
        \cap \BV) \left( \reali; \reali^2 \right) \right\}
    \end{displaymath}
    containing all functions with sufficiently small total variation in $x>0$
    and $x<0$ and semigroups
  \begin{displaymath}
    S_t^{H^-} \colon \mathcal{D} \to \mathcal{D}
  \end{displaymath}
  defined for all $H^-$ sufficiently close to $H_0^-$,
  such that
  \begin{enumerate}[1)]
  \item \label{it:semigroup} for all $t,s\geq 0$ and
    $u\in\mathcal{D}$
    \begin{displaymath}
    S_0^{H^-}u=u,\qquad S_t^{H^-}S_s^{H^-} u = S_{t+s}^{H^-}u;
  \end{displaymath}
\item \label{it:Lipschitz}  
  for all $u,v\in \mathcal{D}$, $H_1^-,\;H_2^-$ in a suitable neighborhood of
  $H_0^-$
  and $t,t' \ge 0$:
    \begin{displaymath}
      \norma{S_t^{H^-_1}u - S_{t'}^{H^-_2}v}_{\L1}
       \leq 
      L \cdot \left\{\norma{u - v}_{\L1} 
      + 
      \modulo{t-t'}+t\cdot |H_1^- - H_2^-|\right\};
    \end{displaymath}
  \item \label{it:tangent} if $u\in\mathcal{D}$ is piecewise constant,
    then for $t$ small, $S_t^{H^-}u$ is the glueing of solutions
    to Riemann problems
    at the points of jump in $u$ and at the weir in $x=0$;
  \item \label{it:solution} for all $u_o \in {\mathcal{D}}$, the map
    $u(t,x) = \left(S_t^{H^-} u_o \right) (x)$ is a weak entropy solution
    to~\eqref{eq:SH}, \eqref{eq:coupl} (see
    \cite[Definition 4.1]{BressanLectureNotes},
    \cite[Definition 2.1]{ColomboGuerraB}).
  \end{enumerate}
  \noindent $S$ is uniquely characterized by~\ref{it:semigroup}),
  \ref{it:Lipschitz}) and~\ref{it:tangent}). 
\end{theorem}
\begin{proof}
  Following \cite[Proposition 4.2]{ColomboGuerraHertySchleper}, the $2\times 2$
  system \eqref{eq:SH} defined for $x\in\reali$ can be rewritten as the following
  $4\times4$ system defined for $x\in\reali^+$:
  \begin{equation}
    \label{eq:systemU}
    \begin{cases}
      \partial_t U + \mathcal{F}(U) =0  & (t,x)\in \reali^+\times \reali^+\\
      b\left(U(t,0+)\right)=g(t) & t\in \reali^+.
    \end{cases}
  \end{equation}
  the relation between $U$ and $u=(h,v)$, between $\mathcal{F}$ and the flow
  in \eqref{eq:SH} being:
  \begin{equation}
    \label{eq:newflux}
    U(t,x) =\left[
    \begin{array}{c}
      h(t,-x)\\
      -(hv)(t,-x)\\
      h(t,x)\\
      (hv)(t,x)     
    \end{array}
  \right]
  \qquad
  \mathcal{F}(t,x)=
  \left[
    \begin{array}{c}
      U_2\\
      \frac{U_2^2}{U_1}+\frac 12 g U_1^2\\
      U_4\\
      \frac{U_4^2}{U_3}+\frac 12 g U_3^2
    \end{array}
    \right]
  \end{equation}
  with $x\in\reali^+$; whereas the boundary conditions becomes
  \begin{displaymath}
    g(t)\dot=\left(
      \begin{array}{c}
        H^--H_0^-\\
        0
      \end{array}
    \right),\quad 
    b\left(U\right)\dot=
    \left(
      \begin{array}{c}
        U_1 - \left(-\frac{U_2}{C}\right)^{\frac 23}-H_0^-\\
        U_4 + U_2
      \end{array}
    \right).
  \end{displaymath}
  The thesis now follows from \cite[Theorem
  2.2]{ColomboGuerraB}. Indeed the assumptions $(\mathbf{\gamma})$,
  (\textbf{b}) and (\textbf{f}) are therein satisfied. More precisely
  the eigenvalues
  of the Jacobian of the flow in \eqref{eq:newflux} are
  \begin{displaymath}
    \frac{U_2}{U_1}\pm\sqrt{gU_1},\qquad \frac{U_4}{U_3}\pm\sqrt{gU_3};
  \end{displaymath}
  therefore in the subsonic region exactly two are positive and
  exactly two are negative. Since here $\gamma(t)=0$, $\dot\gamma(t)=0$,
  condition ($\mathbf{\gamma}$) in \cite[Theorem
  2.2]{ColomboGuerraB} is satisfied with $\ell=2$. Concerning condition
  (\textbf{b}), the positive eigenvalues with the corresponding eigenvectors
  evaluated
  at $\overline U = (h_0,-h_0v_0,h_1,h_1v_1)$ are
  \begin{displaymath}
    \Lambda_3=-v_0 + \sqrt{gh_0},\quad
    R_3=\left(
      \begin{array}{c}
        1\\
        -v_0+\sqrt{gh_0}\\
        0\\
        0
      \end{array}
\right)
  \end{displaymath}
  \begin{displaymath}
    \Lambda_4=v_1 + \sqrt{gh_1},\quad
    R_4=\left(
      \begin{array}{c}
        0\\
        0\\
        1\\
        v_1+\sqrt{gh_1}
      \end{array}
\right).
\end{displaymath}
Conditions \eqref{eq:theomhy} imply $b\left(\overline U\right)=0$, and
\begin{displaymath}
  \left[Db(\overline U) R_3,\;Db(\overline U) R_4\right]=
  \left[
    \begin{array}{cc}
      1+\frac{2}{3C}\left(\sqrt{gh}-v_0\right)\cdot
      \left(-\frac{U_2}{C}\right)^{-\frac 13}
      &0\\
      -v_0+\sqrt{gh_0}& v_1+\sqrt{gh_1}
    \end{array}
    \right].
  \end{displaymath}
  The determinant of the above matrix is given by
  \begin{displaymath}
    \left(1+\frac{2}{3C}\left(\sqrt{gh}-v_0\right)\cdot
    \left(-\frac{U_2}{C}\right)^{-\frac 13}\right)
    \left(v_1+\sqrt{gh_1}\right)
  \end{displaymath}
  which is strictly positive since $(h_0,v_0)$ and $(h_1,v_1)$ belong to the
  subsonic region.
  Thus condition (\textbf{b}) is satisfied.
  Concerning condition (\textbf{f}), system \eqref{eq:systemU} is not
  necessarily strictly 
  hyperbolic, for it is obtained glueing two copies of system \eqref{eq:SH}.
  Nevertheless, the two systems are coupled only through the boundary conditions,
  hence the whole wave front tracking procedure in the proof of
  \cite[Theorem 2.2]{ColomboGuerraB} applies.
  Concerning the Lipschitz dependence on the height $H^-$, observe
  that as in \cite[Theorem 2.2]{ColomboGuerraB} we have for 
  \begin{displaymath}
    g(t)= H_1^- - H_0^-, \qquad \bar g(t)= H_2^- - H_0^-,
  \end{displaymath}
  and hence
  \begin{displaymath}
    \int_0^t\left|g(\tau)-\bar g(\tau)\right|\, d\tau = t\cdot
    \left|H_1^- - H_2^-\right|.
  \end{displaymath}
\end{proof}
\section{Computational results}
We present some numerical results on pooled steps using a
finite--volume method in the conservative variables to solve for the
system dynamics in each canal. The coupling conditions
(\ref{eq:coupl}) induce boundary conditions for each canal at each
time--step.  For given data $U_{0}^{\pm}:=(h^{\pm}_{0},
(hv)^{\pm}_{0})$ close to the weir the conditions yield the boundary
states at each connected canal. In the numerical computation of the
boundary states we proceed as in Section \ref{sec:RP} using Newton's method
applied to (\ref{eq:coupl}) where $h^{\pm}$ and $(hv)^{\pm}$ are given
by the forward (backwards) 1-(2) Lax--wave curves through the initial
state $h^{\pm}_{0}, (hv)^{\pm}_{0}$.

According to the solution of the Riemann Problem described in
Section \ref{sec:RP} there are three possible scenarios. If
$h^\pm\le H^\pm$ the coupling condition \eqref{eq:coupl} reduce to
$v^\pm =0$ and no water passes the weir. If $h^->H^-$ and $h^+\le H^+$
the water is overflowing the weir from the left to the right. 
If $h^+>H^+$ and $h^-\le H^-$
the water is overflowing the weir from the right to the left.
If $h^+>H^+$ and $h^-> H^-$
the water is flowing from the left or from the right depending on the
sign of $(h^--H^-) - (h^+-H^+)$.

Now, we present a numerical result for overflowing three connected
pooled--steps. Each canal has length $L=1$ and the height of each weir
is 1.5.  We simulate two situations. In the first case the initial
water level in each canal is low (equal to $H^{-})$ in the second case
the canals are already full (height is equal to $H^{+}$). In both case
the water initial is still $v=0$ and a wave with a height $h=2\;
H^{+}$ and $hv = 5$ enters on the first canal. This wave lasts until
$T=20$ and is then followed by a wave of height $h=H^{-}$ and zero
flux. The simulation time is $T=60$ for both scenarios.  We present
snapshots of height and velocity at different times. The solid lines
are the weir, the dotted line is ground level. It is a pooled step and
therefore the first canal is on a higher level above ground than the
last one.
\par In the first scenario (see Figure \ref{fig05}) the wave enters
the system of pooled steps and overflows the connected canals. After
this wave passed the system slowly reaches again an almost steady
state. In the second scenario (see Figure \ref{fig06}) we observe
initial dynamics due to the overflow of the pooled steps. On the
second canal these dynamics interact with the incoming wave. After the
wave passed the system slowly reaches a steady state with heights
below critical (i.e., $h \leq H^{+}$) in the first and second
canal. The water is still flowing over the third weir in this
simulation.

\paragraph{Acknowledgments} \par
  We thank Holger Sch\"uttrumpf, Institut f\"ur Wasserwirtschaft, RWTH
  Aachen University, for pointing out this interesting problem to
  us. The work has been supported by the German Research Foundation
  HE5286/6-1 and DAAD 50727872 and by the Vigoni project 2009.
  The first and third authors acknowledge
  the warm hospitality of RWTH Aachen University where part of this work was 
  completed. 

\newpage
\begin{figure}[h!]
\center
\epsfig{file=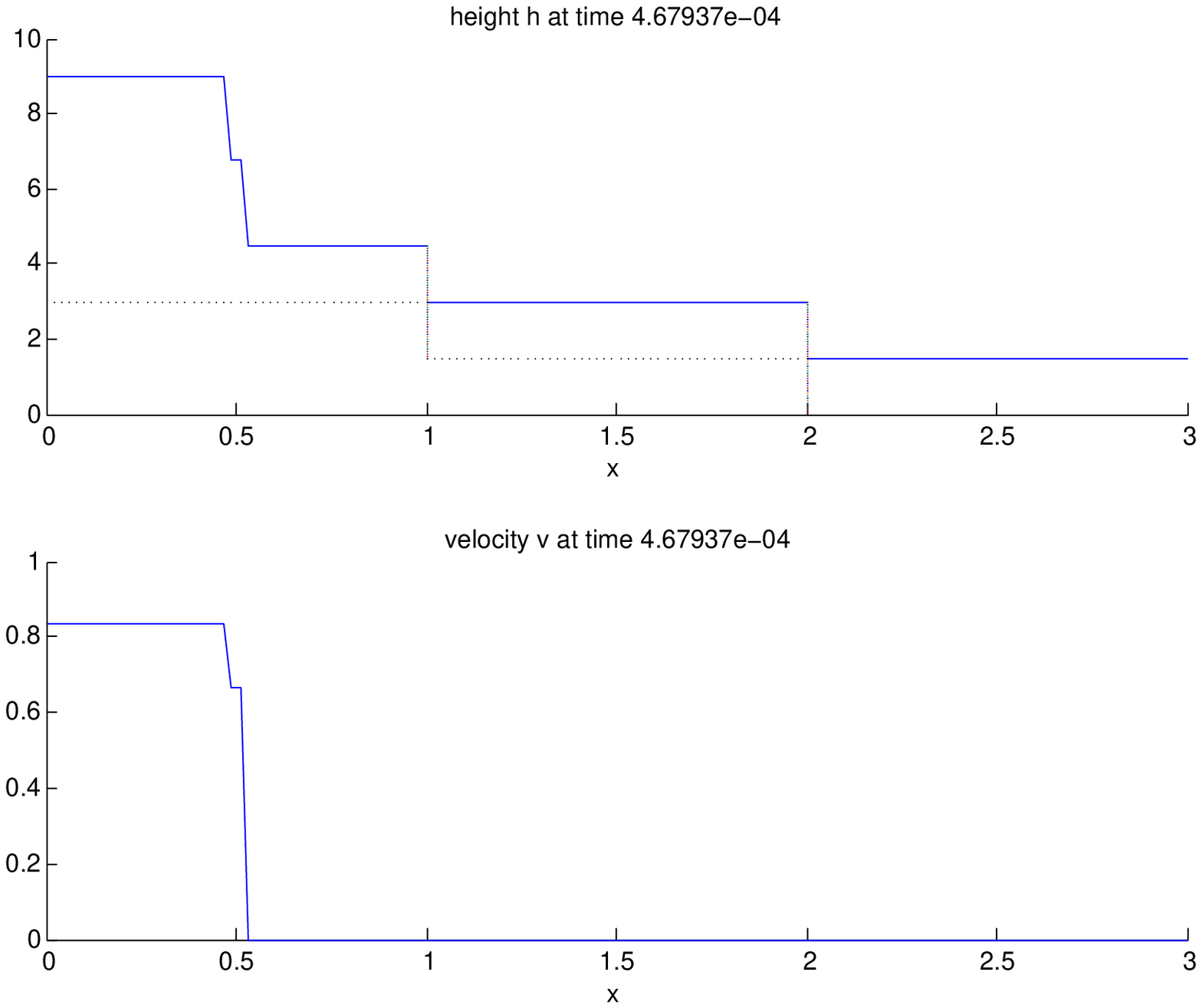,width=.3\textwidth}
\epsfig{file=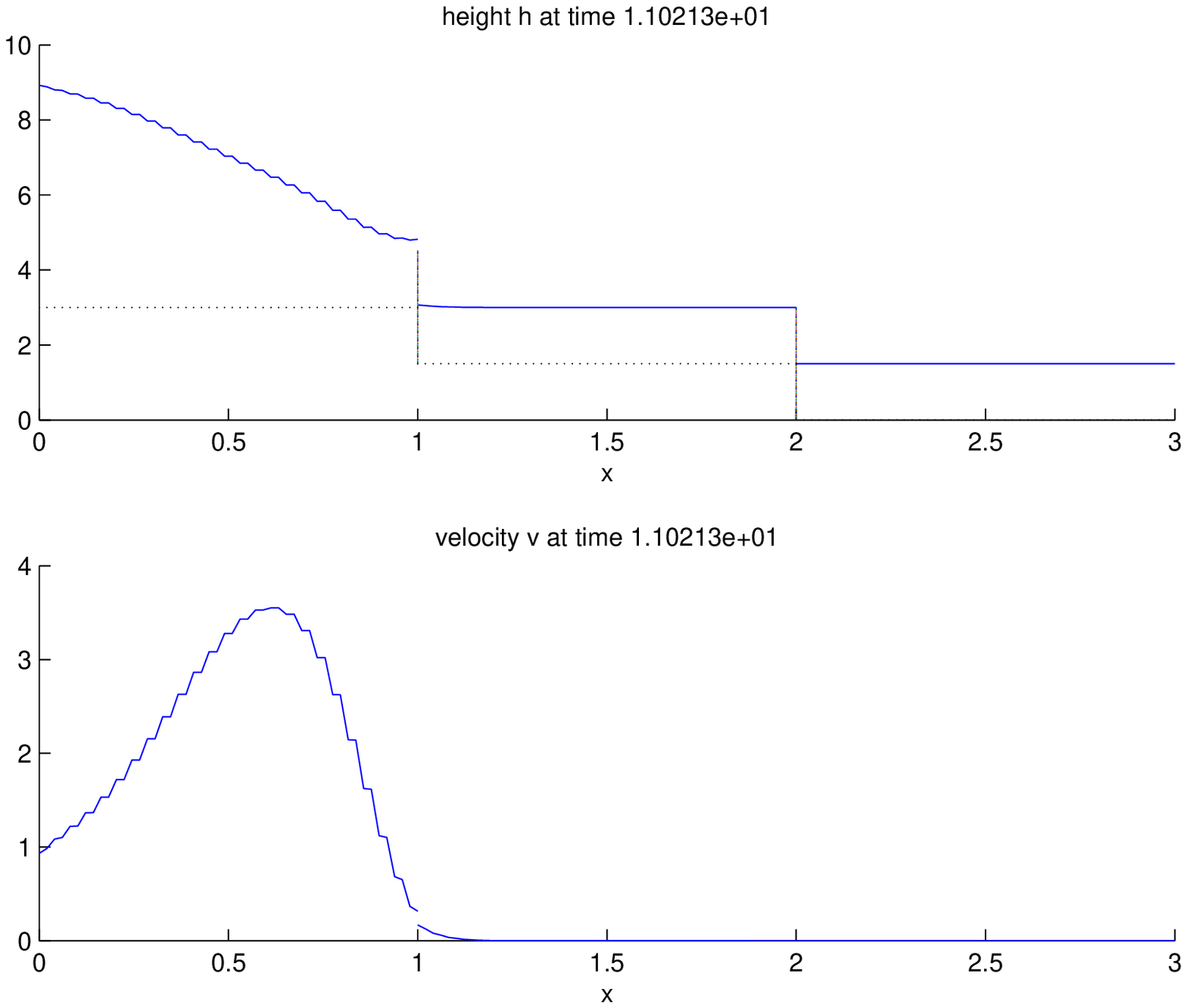,width=.3\textwidth}
\epsfig{file=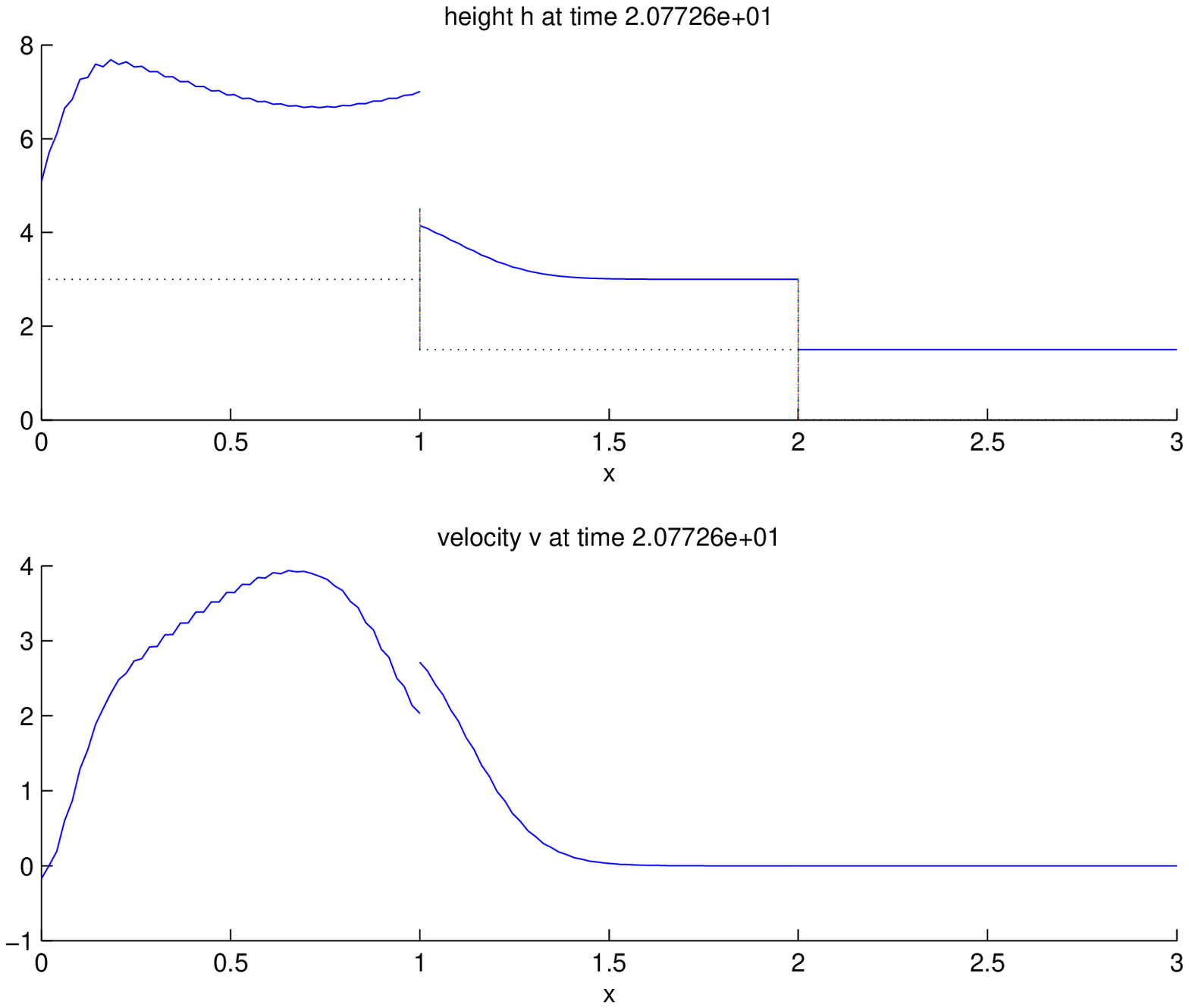,width=.3\textwidth} \\[10mm]
\epsfig{file=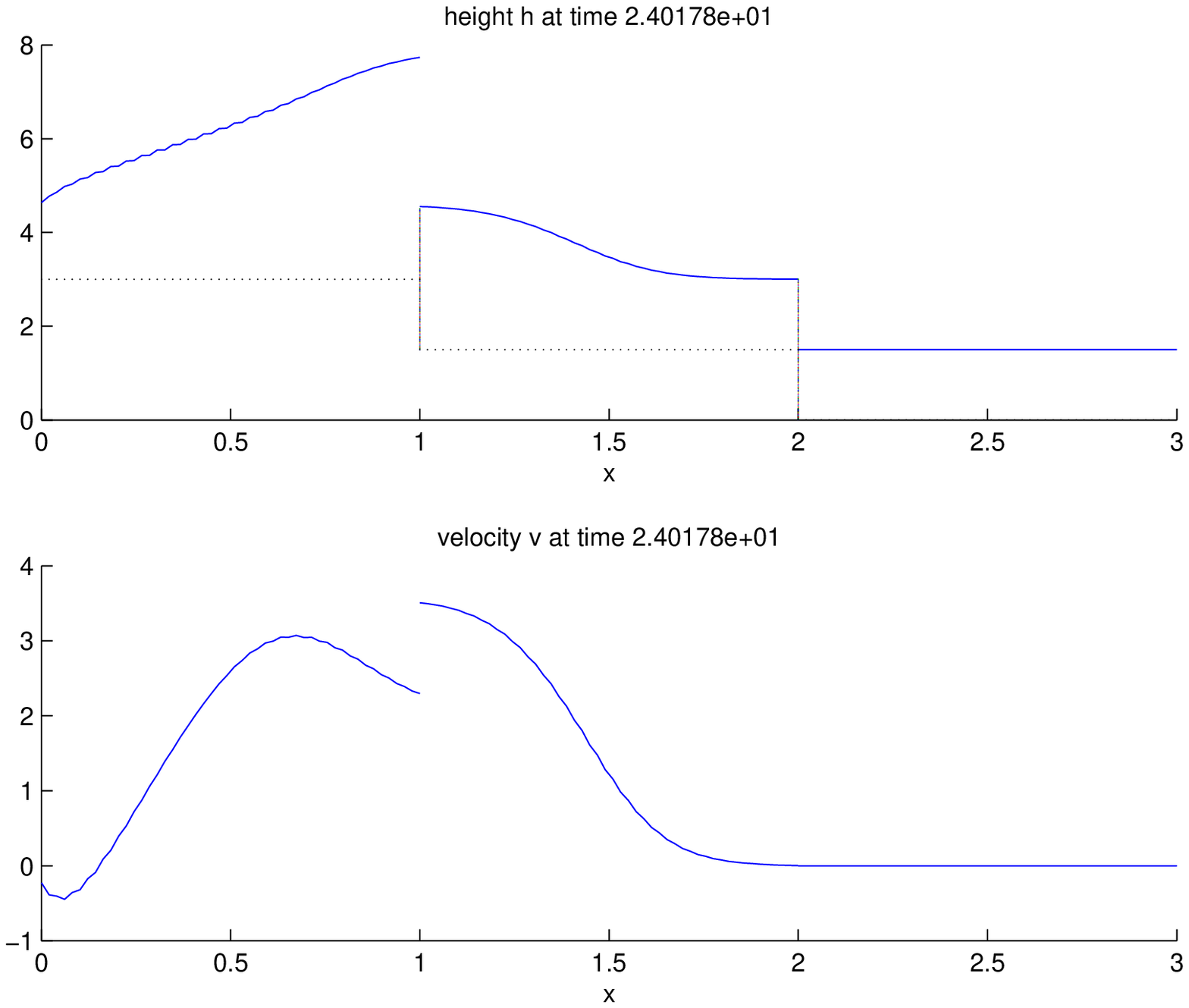,width=.3\textwidth}
\epsfig{file=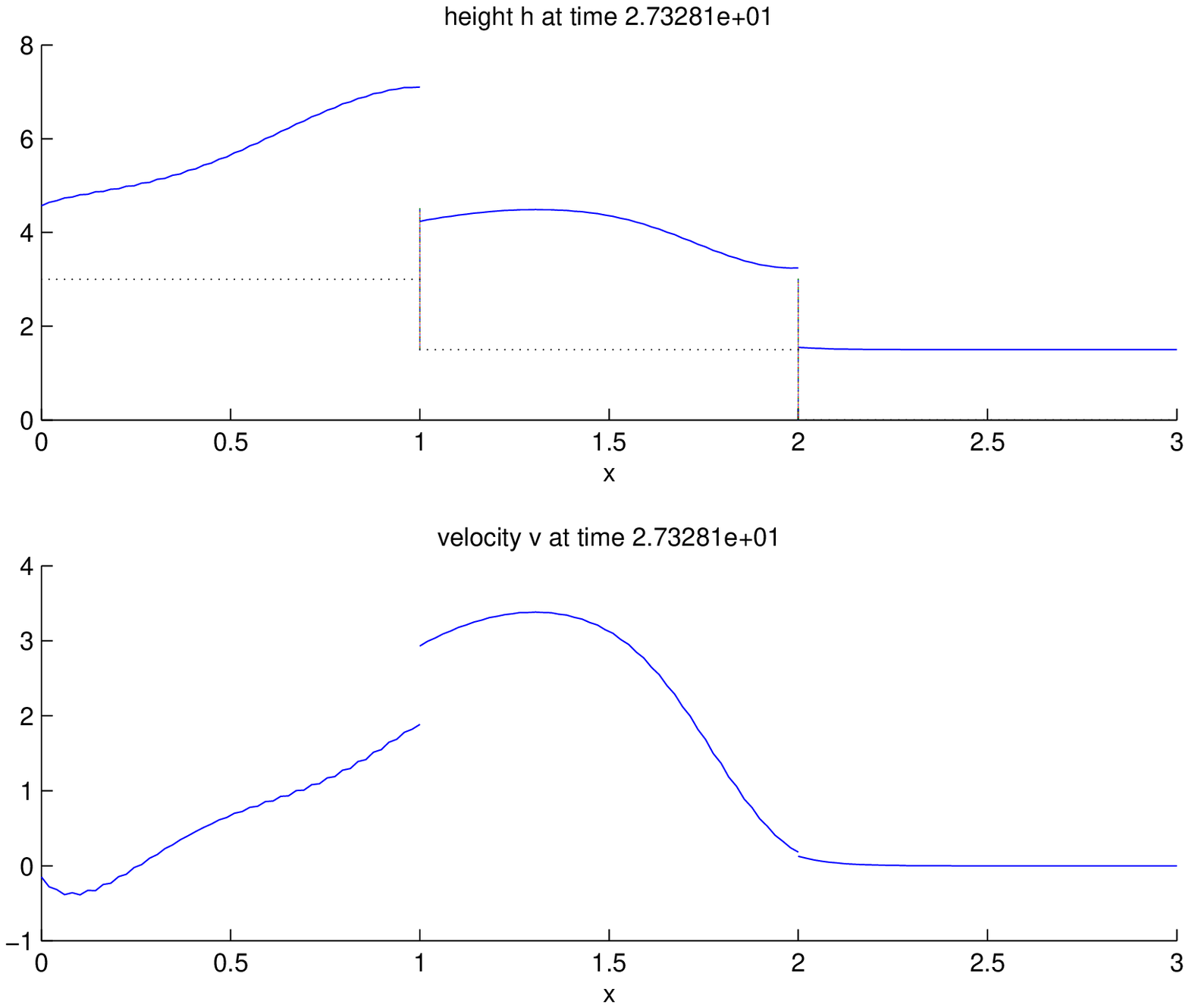,width=.3\textwidth} 
\epsfig{file=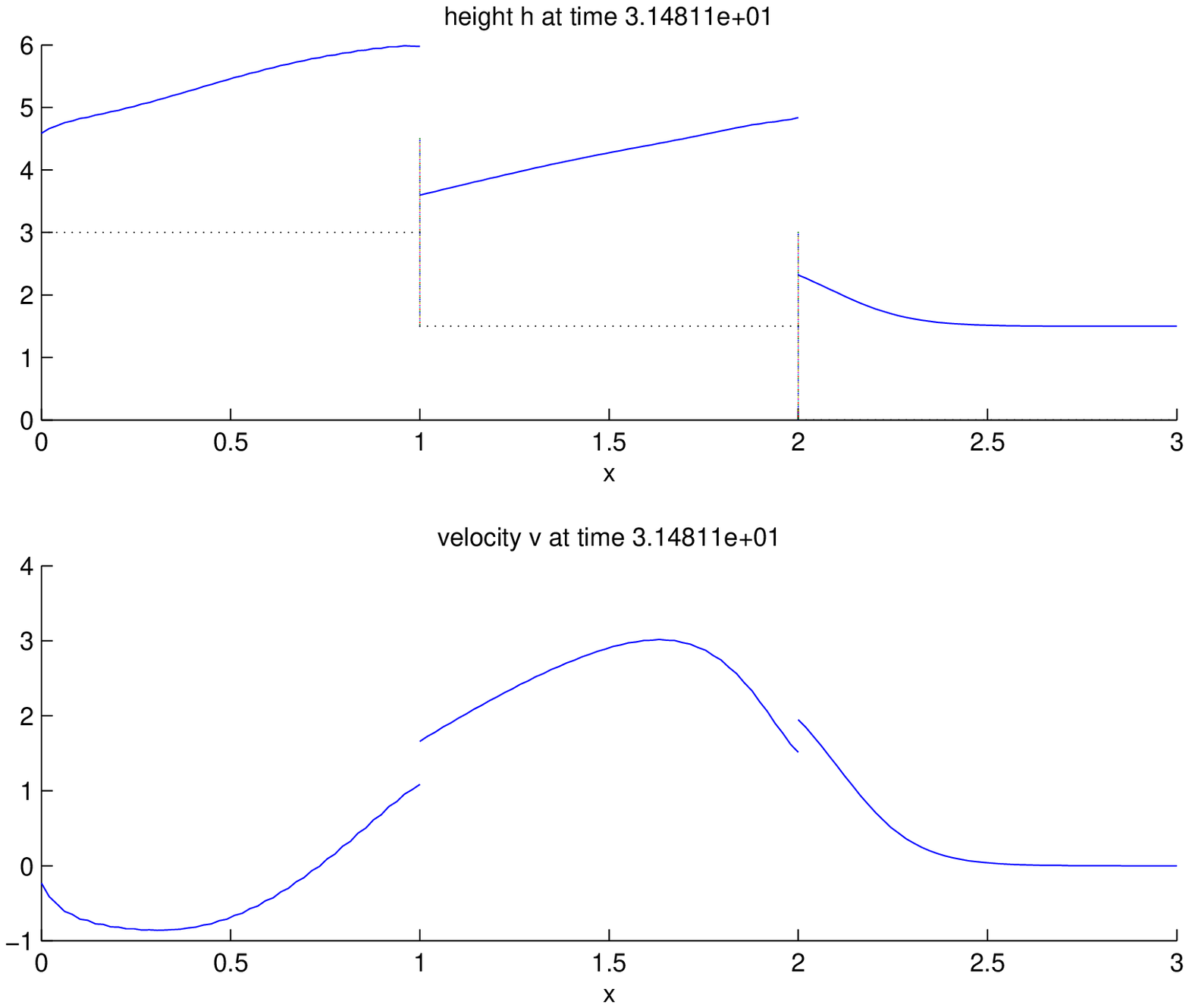,width=.3\textwidth} \\[10mm]
\epsfig{file=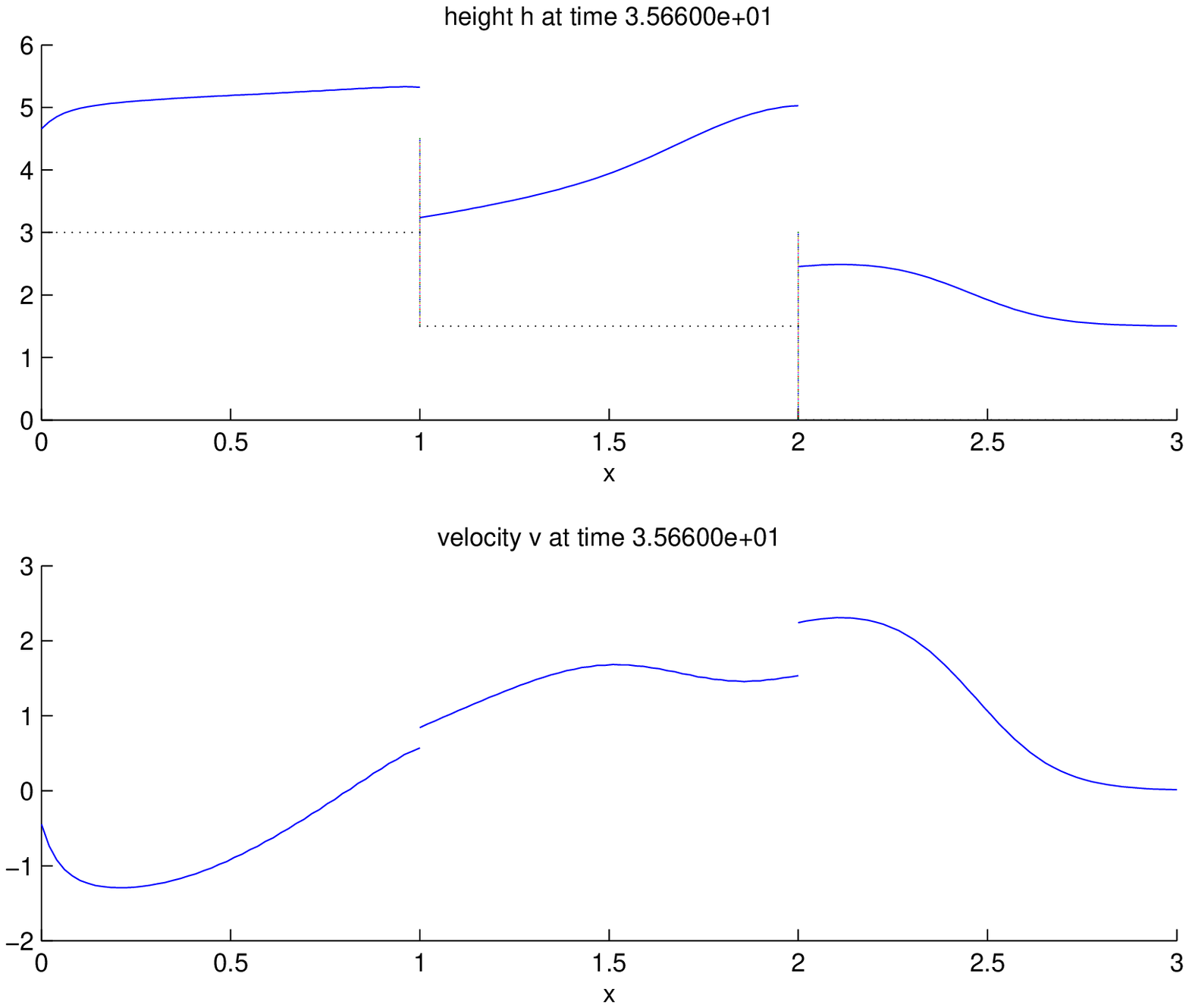,width=.3\textwidth}
\epsfig{file=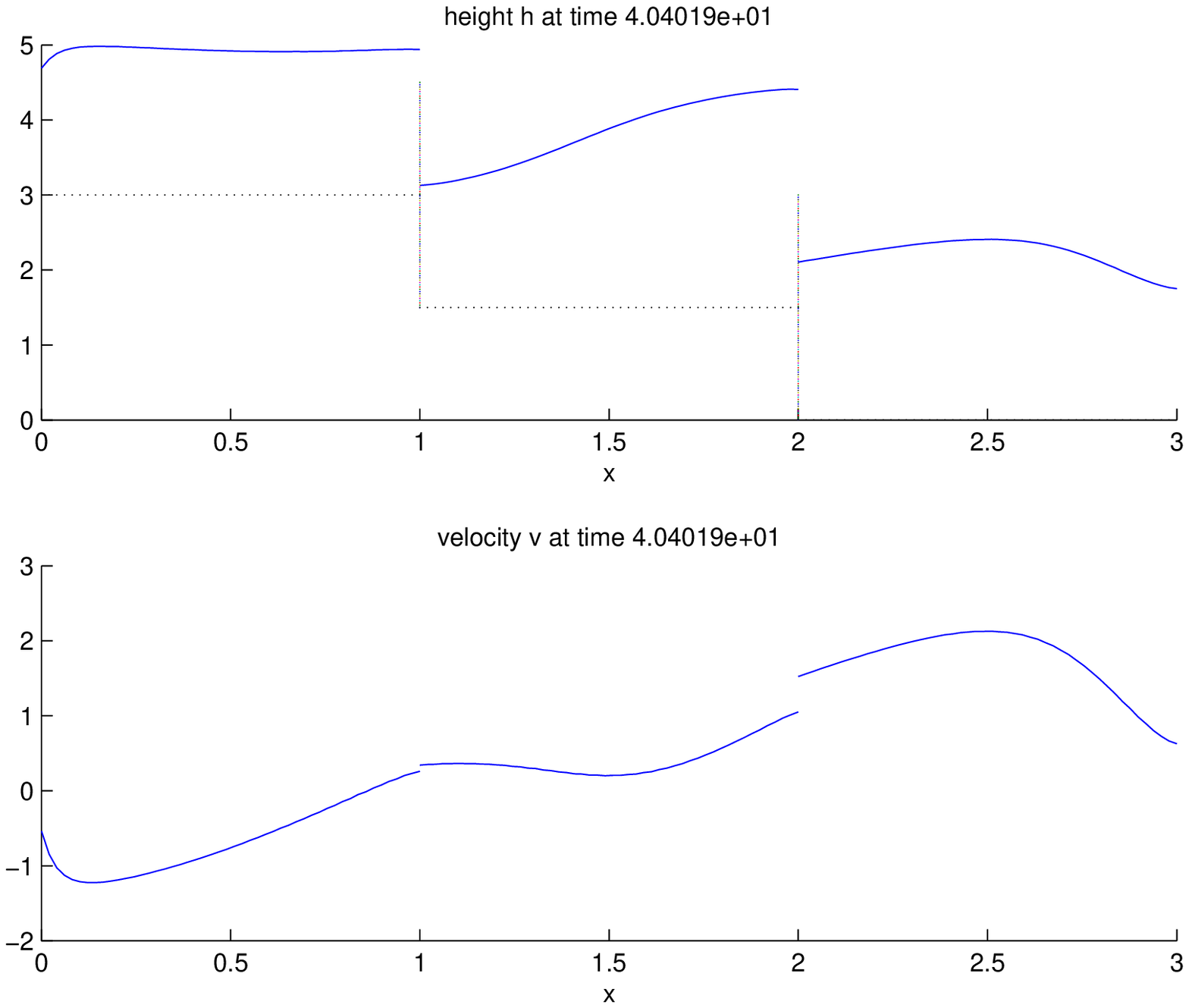,width=.3\textwidth} 
\epsfig{file=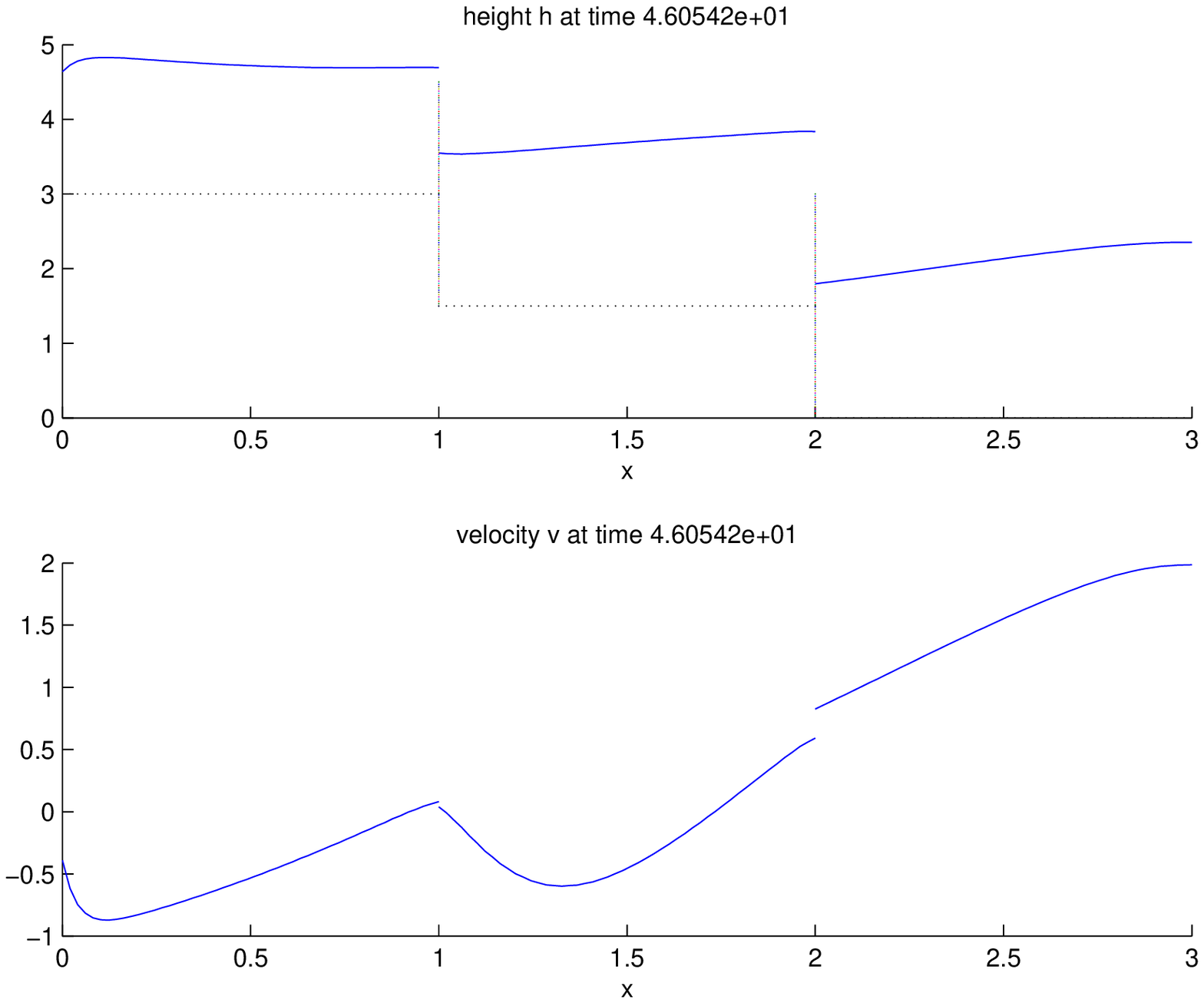,width=.3\textwidth} \\[10mm]
\epsfig{file=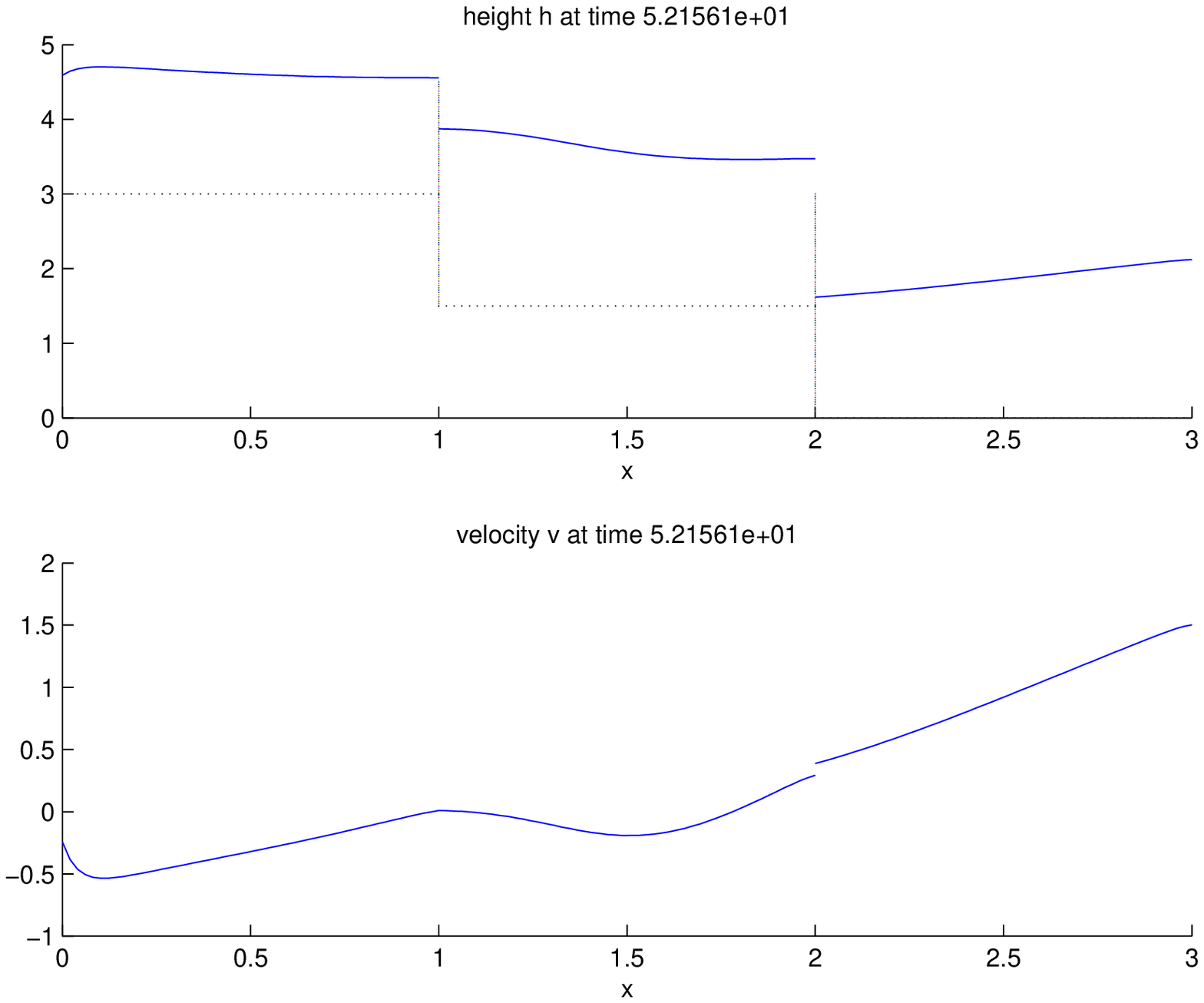,width=.3\textwidth}
\epsfig{file=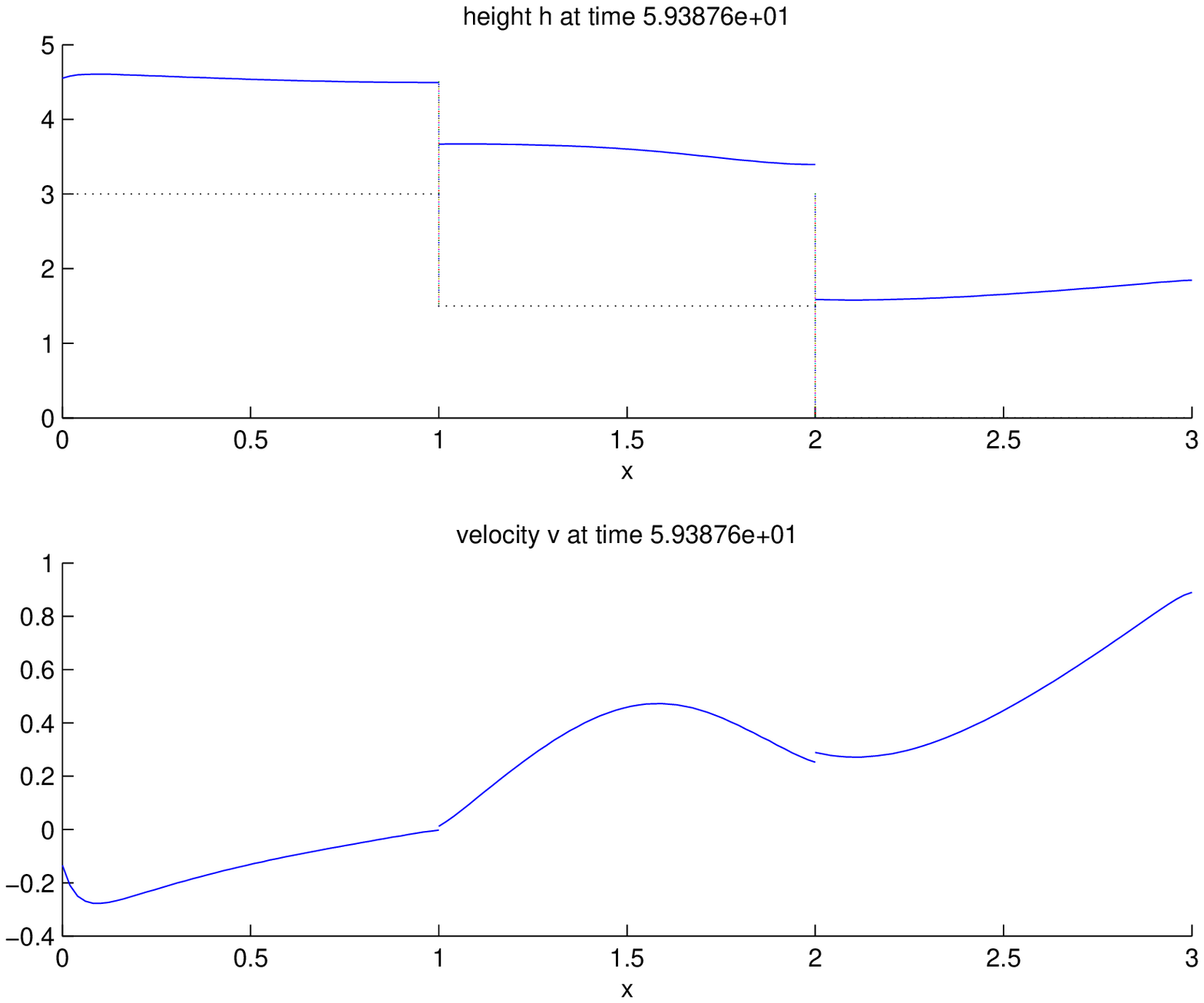,width=.3\textwidth} 
\caption{ Simulation results for a strong wave entering the pooled step. Initially, the water level on each step 
is below critical. The solution $(h,v)$ at different times is depicted.  The time increases from left to right and top to bottom. }
\label{fig05}
\end{figure}
\newpage
\begin{figure}[h!]
\center
\epsfig{file=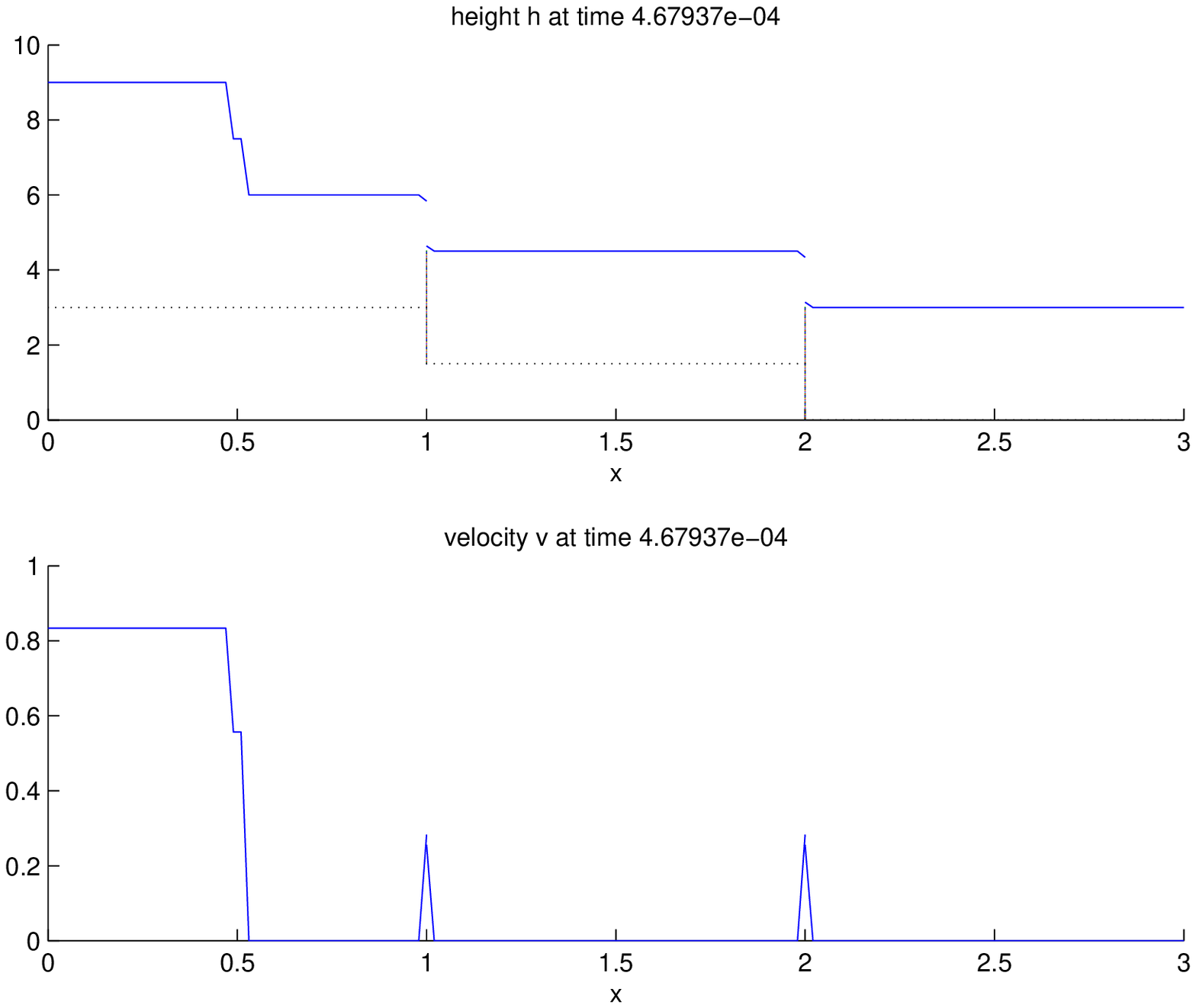,width=.3\textwidth}
\epsfig{file=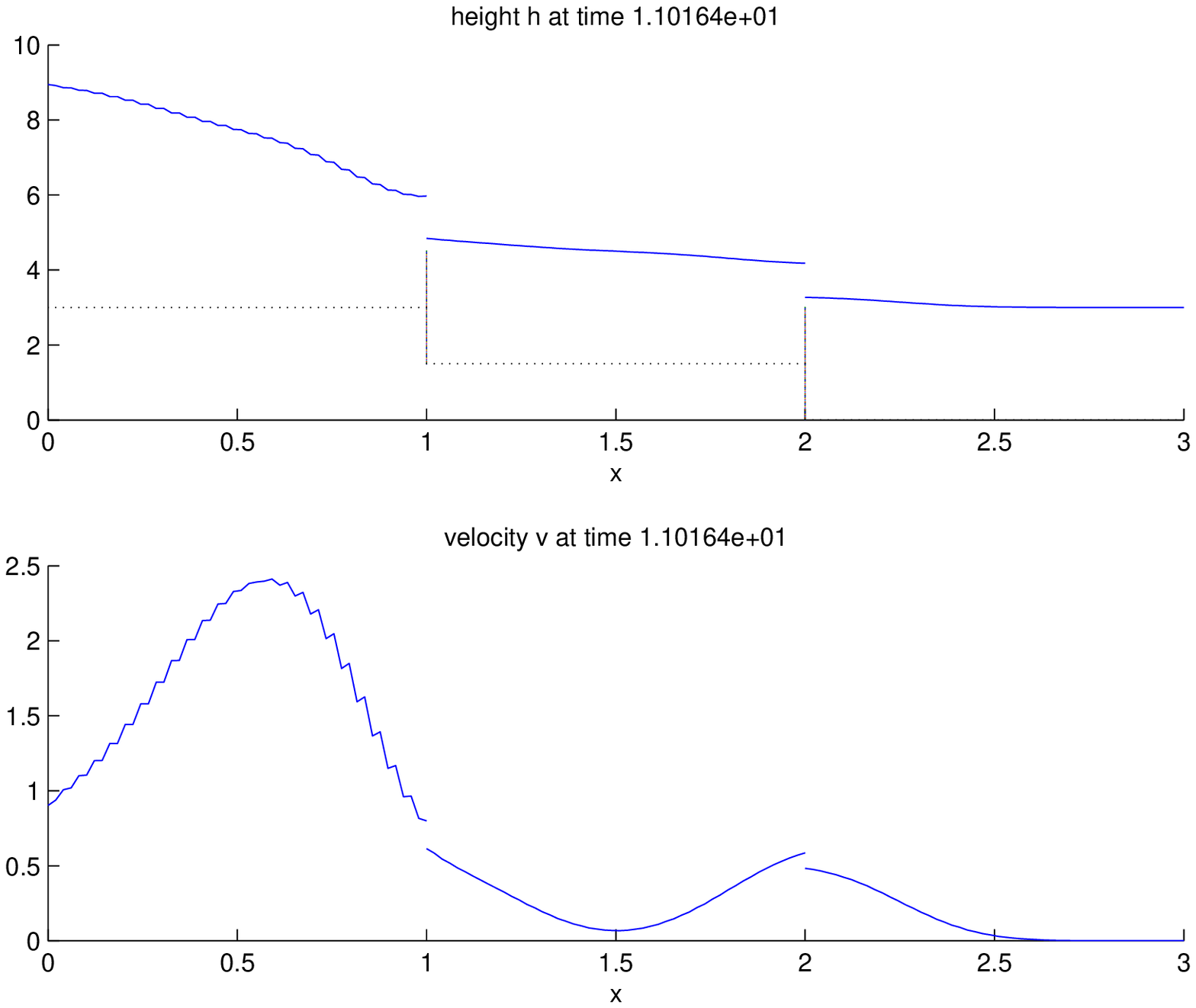,width=.3\textwidth}
\epsfig{file=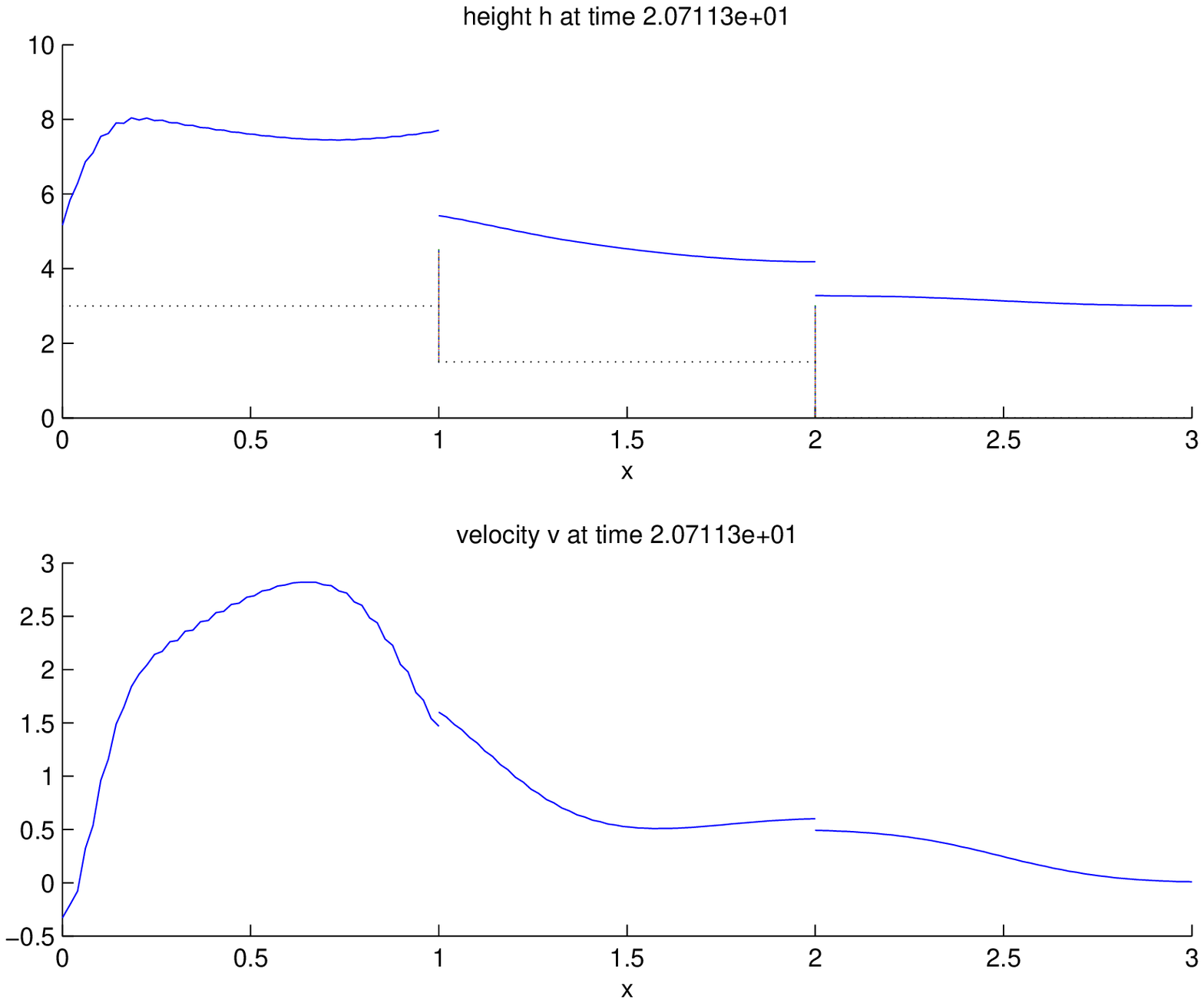,width=.3\textwidth} \\[10mm]
\epsfig{file=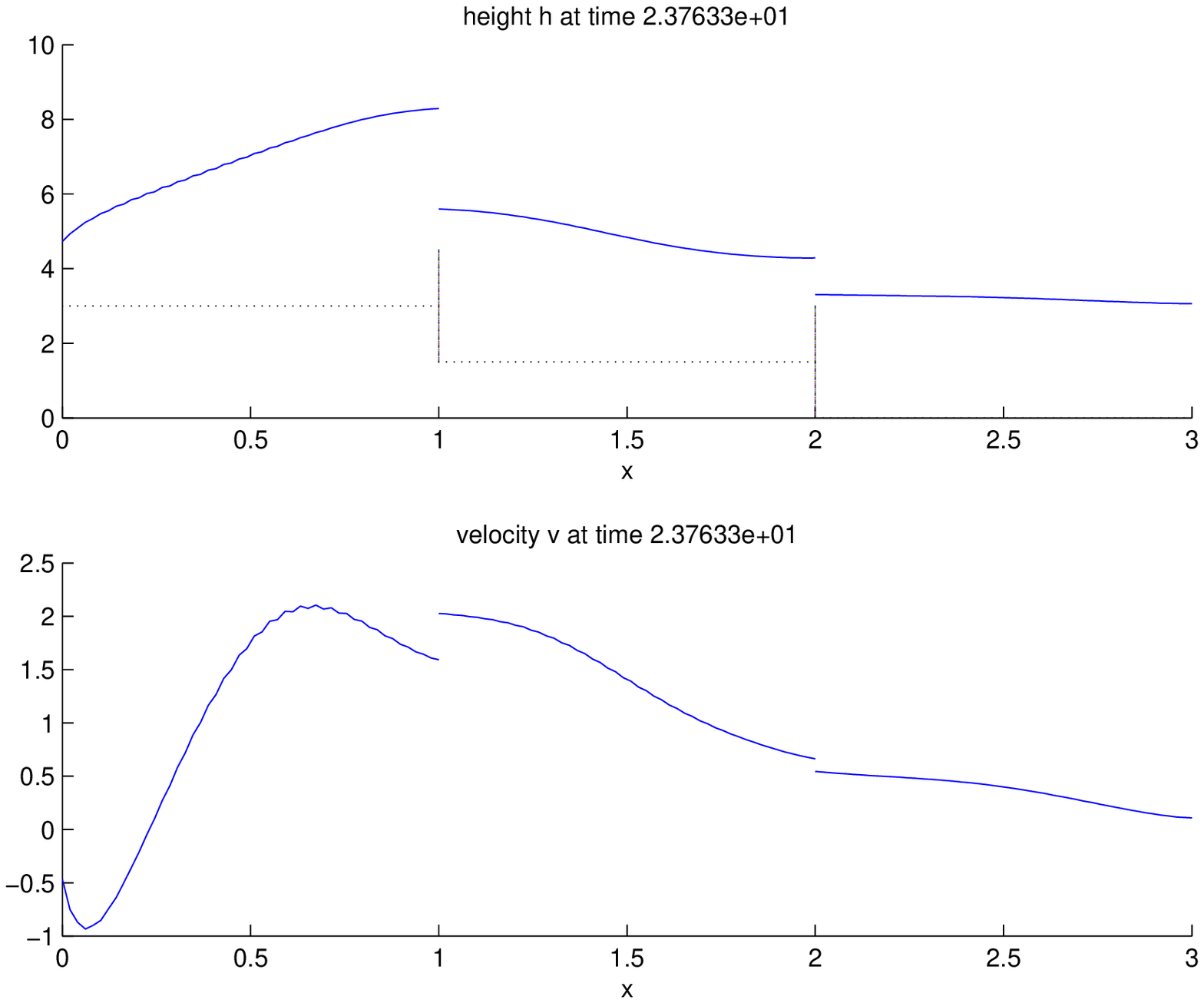,width=.3\textwidth}
\epsfig{file=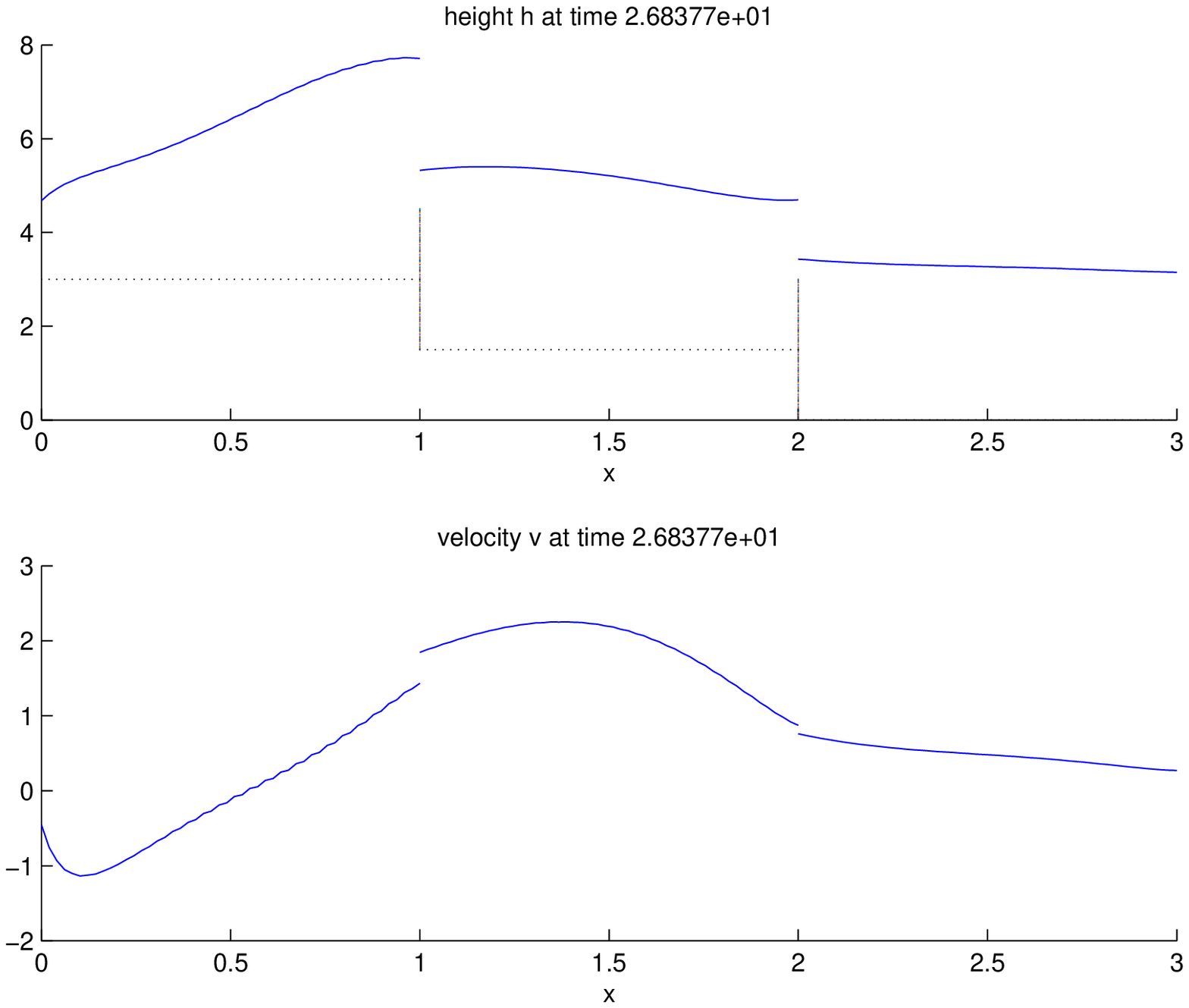,width=.3\textwidth} 
\epsfig{file=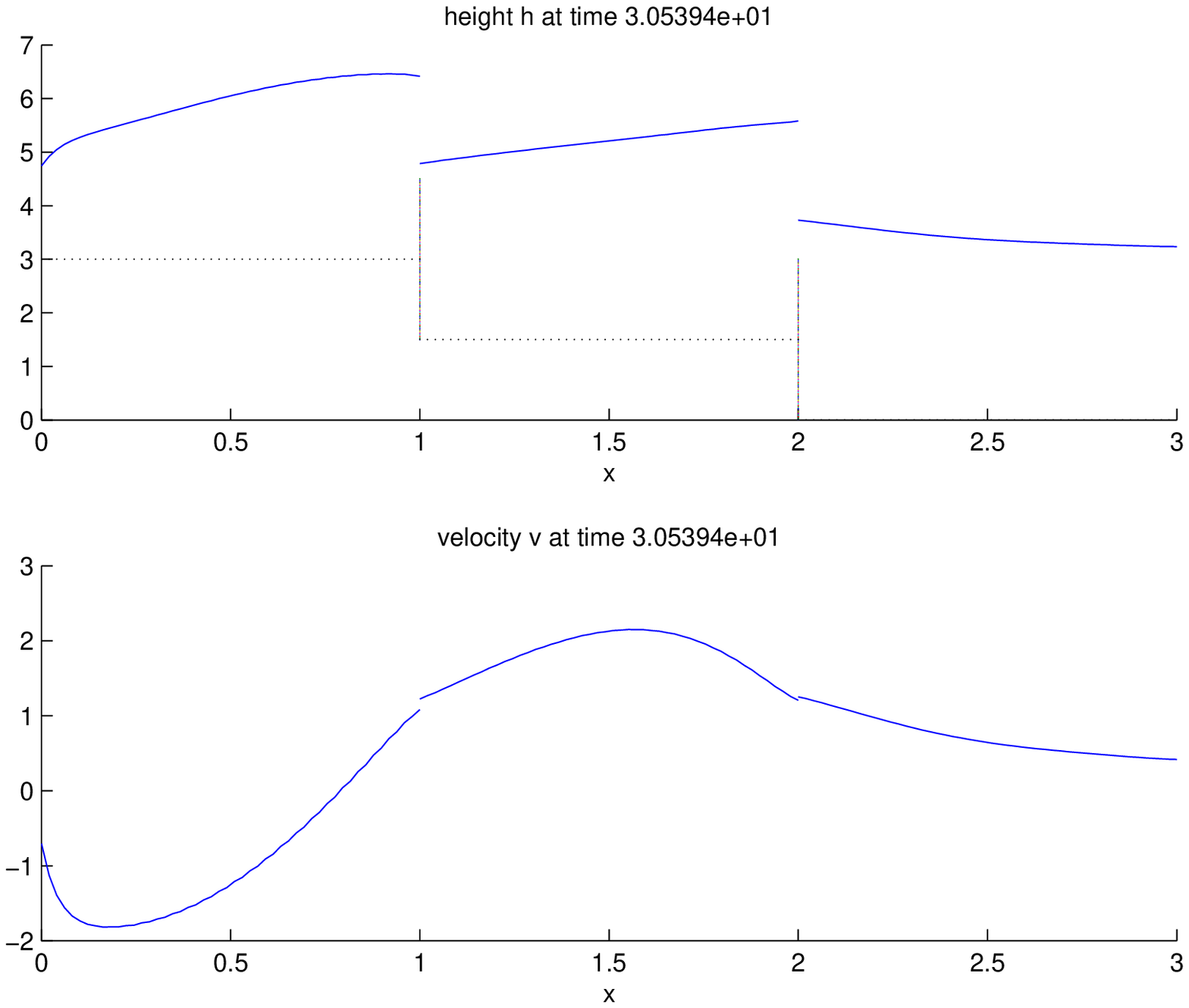,width=.3\textwidth} \\[10mm]
\epsfig{file=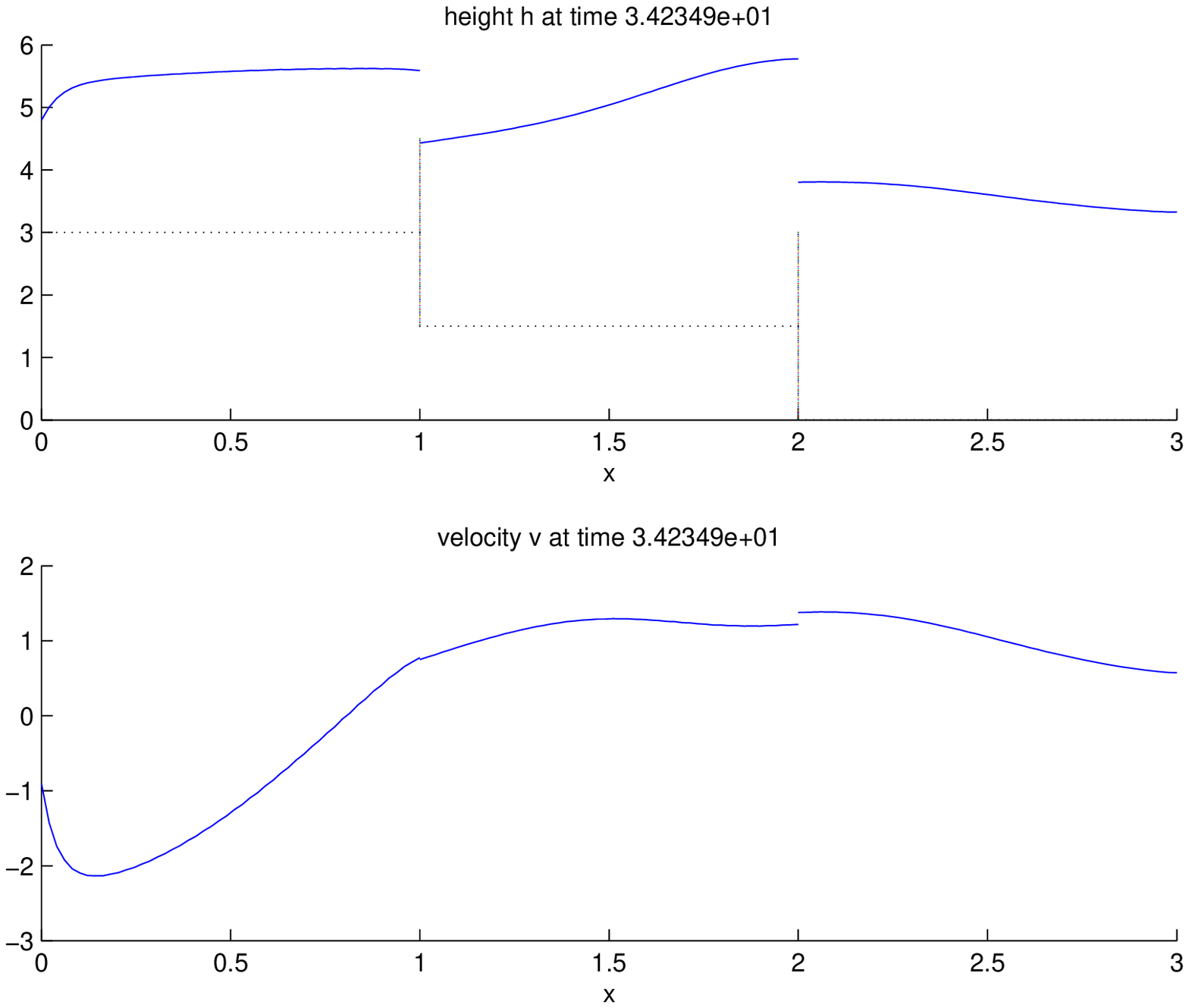,width=.3\textwidth}
\epsfig{file=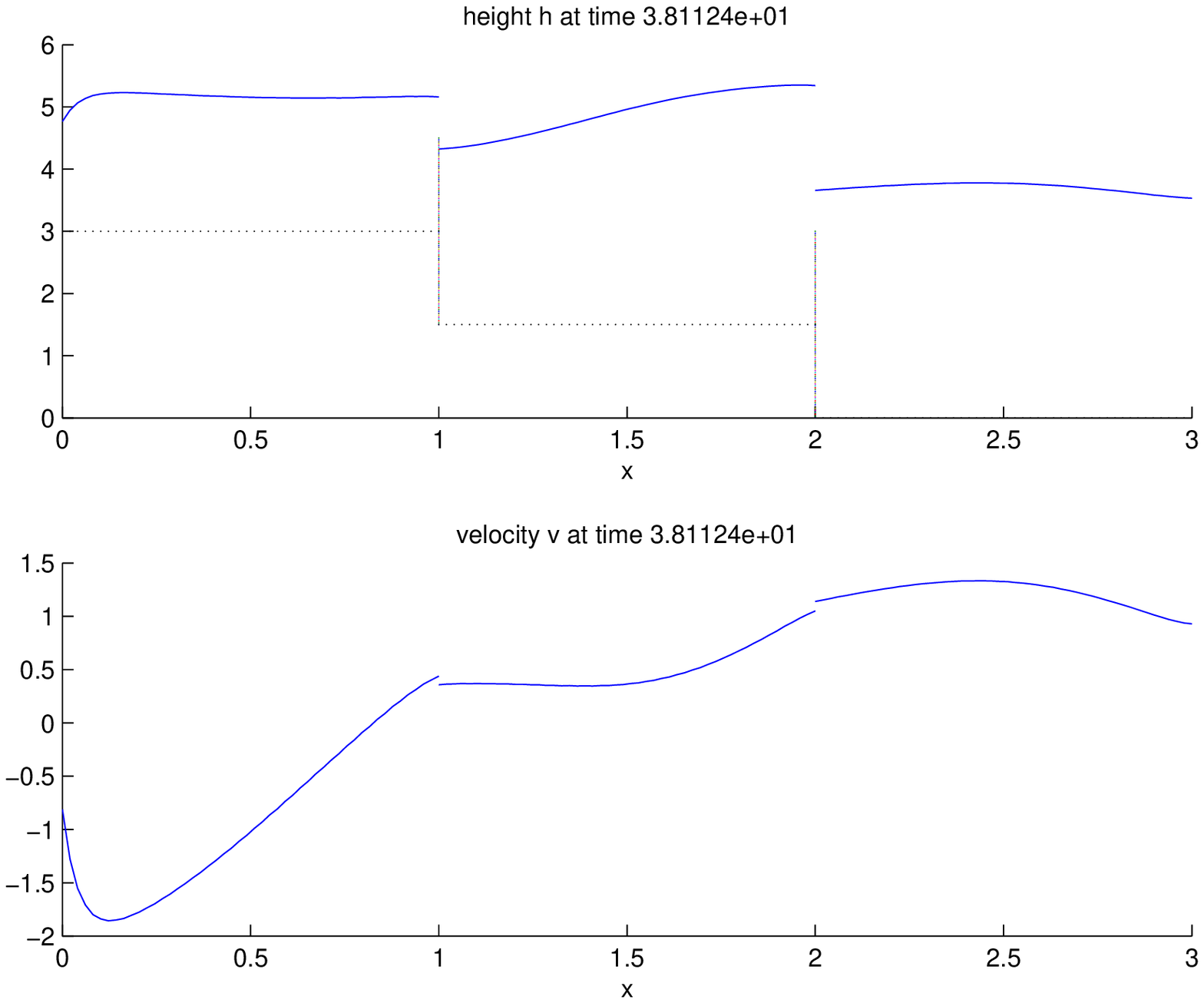,width=.3\textwidth} 
\epsfig{file=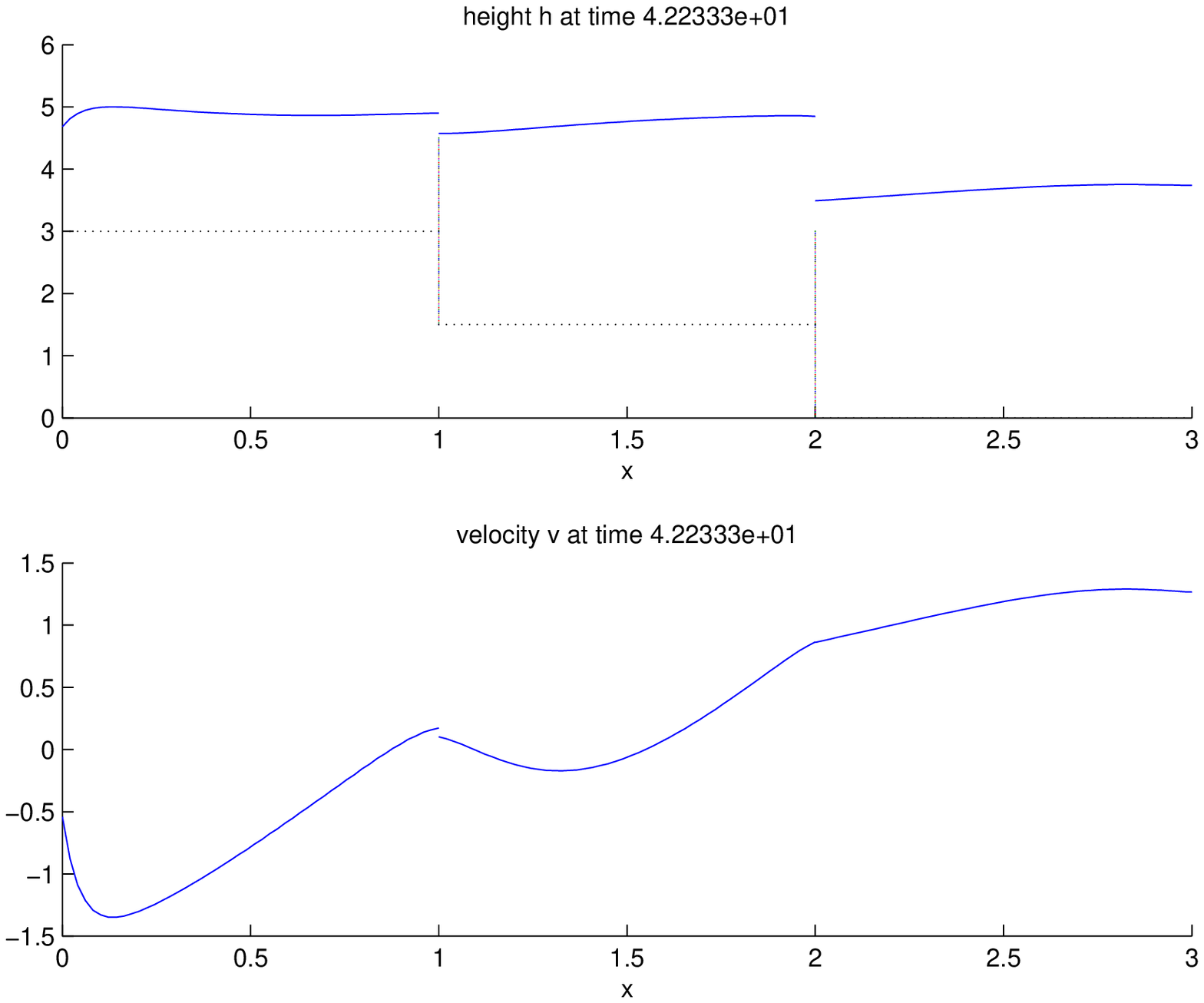,width=.3\textwidth} \\[10mm]
\epsfig{file=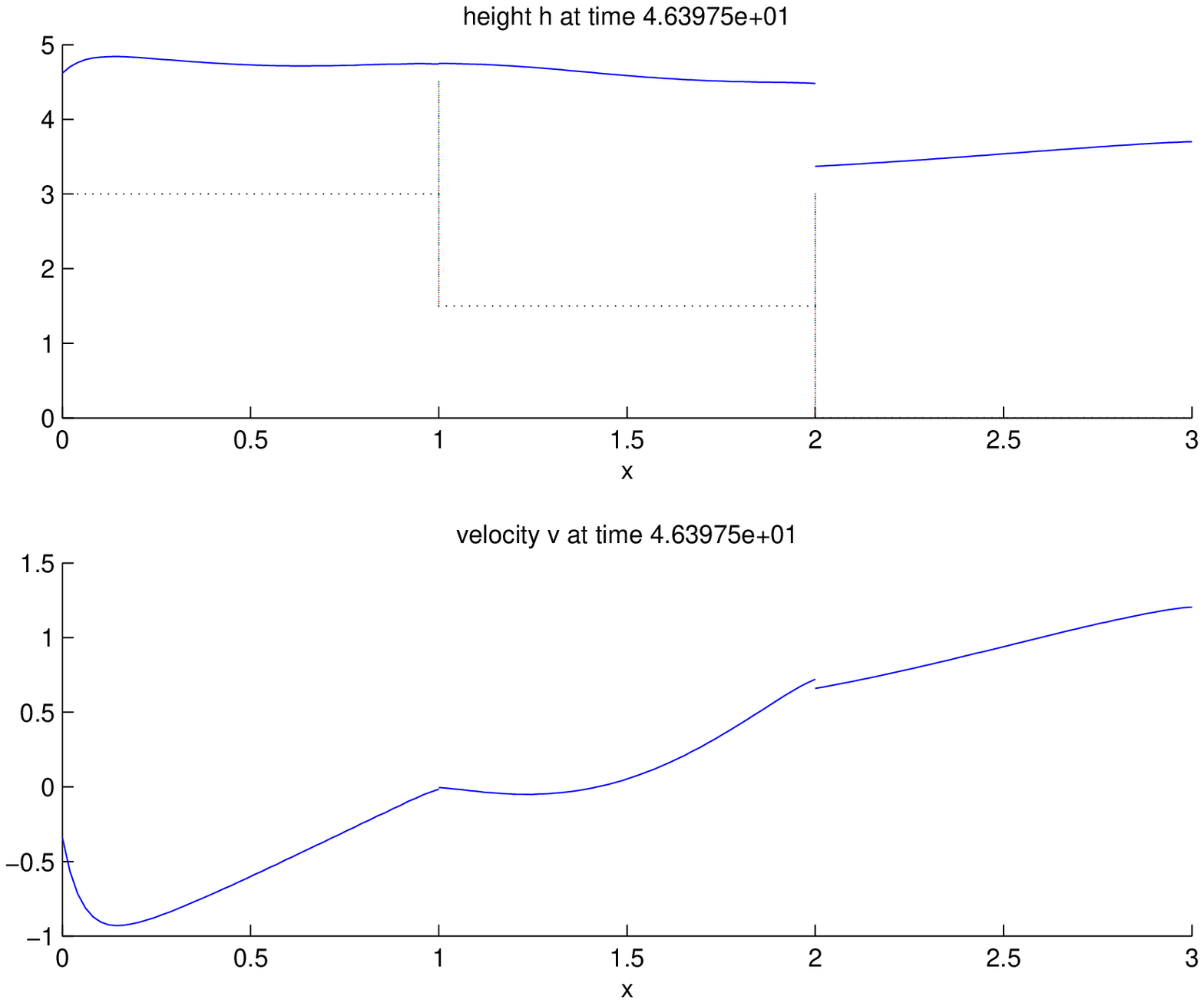,width=.3\textwidth} 
\epsfig{file=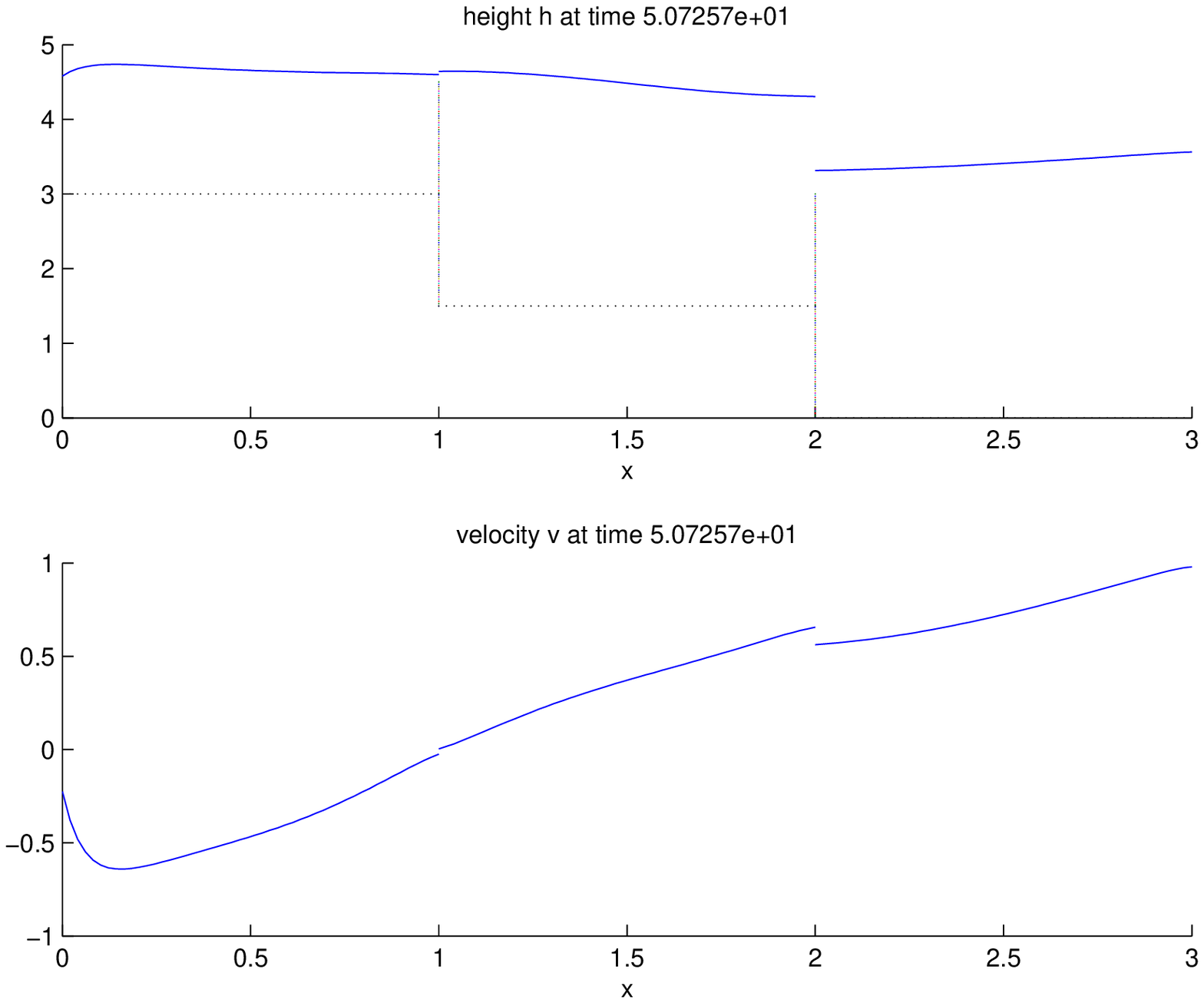,width=.3\textwidth}
\epsfig{file=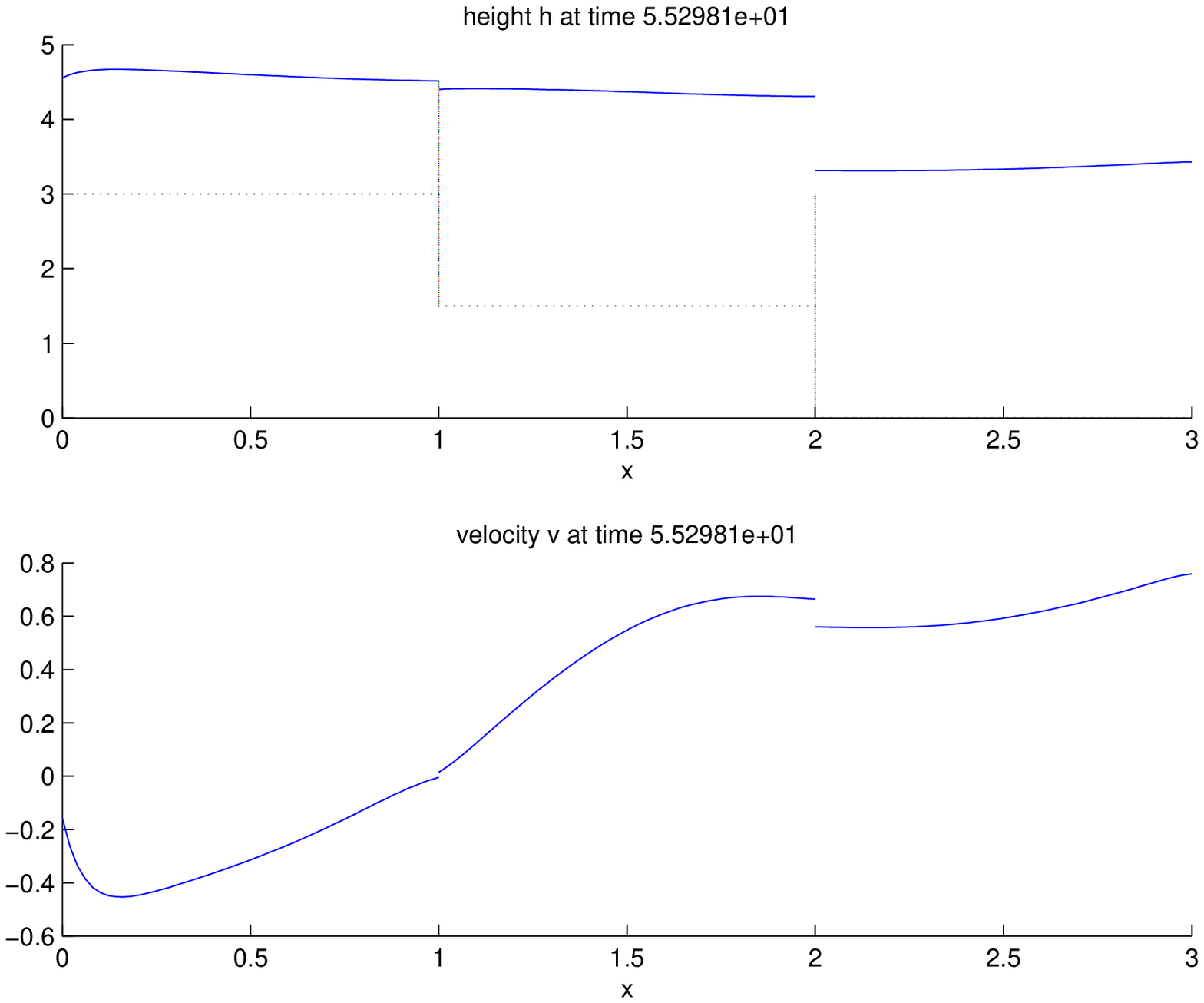,width=.3\textwidth} 
\caption{ Simulation results for a strong wave entering the pooled step. Initially, the water level on each step 
is above the  critical level.  The solution $(h,v)$ at different times is depicted.  The time increases from left to right and top to bottom.}
\label{fig06}
\end{figure}
\newpage

  
{\small{

    \bibliographystyle{abbrv}

    \bibliography{Dighe}

  }}
\end{document}